\documentclass[11pt]{article}
\usepackage[T1]{fontenc}
\usepackage[latin1]{inputenc}
\usepackage{amsmath,amssymb,euscript}
\usepackage{bbm}
\usepackage{graphicx}
\usepackage{epic}
\usepackage{epsfig}
\usepackage{psfrag}
\usepackage{mathrsfs}

\newtheorem{thm}{Theorem}[section]
\newtheorem{lem}[thm]{Lemma}
\newtheorem{cor}[thm]{Corollary}

\newtheorem{defi}[thm]{Definition}
\newtheorem{prop}[thm]{Proposition}
\newtheorem{rem}[thm]{Remark}

\newenvironment{proof}{\noindent {\bf Proof \phantom{9}}}
{\hfill $\square$ \vspace{0.25cm}}

\def\be{\begin{eqnarray}}
\def\ee{\end{eqnarray}}
\def\ben{\begin{eqnarray*}}
\def\een{\end{eqnarray*}}

\numberwithin{equation}{section}
\numberwithin{figure}{section}

\def\be{\begin{eqnarray}}
\def\ee{\end{eqnarray}}

\newcommand{\RR}{\mathbb{R}}

\newcommand{\DD}{\mathbb{D}}
\newcommand{\PP}{\mathbb{P}}

\def\me{\medskip\noindent}
\def\bi{\bigskip\noindent}

\title{\bf Adaptation in a stochastic multi-resources chemostat model}

\author{Nicolas Champagnat\thanks{Universit\'e de Lorraine, Institut Elie Cartan de Lorraine,
    UMR 7502, Vand\oe uvre-l\`es-Nancy, F-54506, France; E-mail:
    \texttt{Nicolas.Champagnat@inria.fr}}~\thanks{CNRS, Institut Elie Cartan de Lorraine, UMR
    7502, Vand\oe uvre-l\`es-Nancy, F-54506, France}~\thanks{Inria, TOSCA,
    Villers-l\`es-Nancy, F-54600, France}, Pierre-Emmanuel Jabin\thanks{Cscamm and Dpt. of
    Mathematics, University of Maryland, College Park, MD 20742 USA,
    E-mail:~\texttt{pjabin@umd.edu}}, Sylvie M\'el\'eard\thanks{CMAP, Ecole Polytechnique,
    CNRS, route de Saclay, 91128 Palaiseau Cedex-France; E-mail:
    \texttt{sylvie.meleard@polytechnique.edu}}}

\date{\today}

\begin{document}

\maketitle

\begin{abstract}
  We are interested in modeling the Darwinian evolution resulting from the interplay of phenotypic variation and natural selection
  through ecological interactions, in the specific scales of the biological framework of adaptive dynamics. Adaptive dynamics so far
  has been put on a rigorous footing only for direct competition models (Lotka-Volterra models) involving a competition kernel which
  describes the competition pressure from one individual to another one. We extend this to a multi-resources chemostat model, where
  the competition between individuals results from the sharing of several resources which have their own dynamics. Starting from a
  stochastic birth and death process model, we prove that, when advantageous mutations are rare, the population behaves on the
  mutational time scale as a jump process moving between equilibrium states (the polymorphic evolution sequence of the adaptive
  dynamics literature). An essential technical ingredient is the study of the long time behavior of a chemostat multi-resources
  dynamical system. In the small mutational steps limit this process in turn gives rise to a differential equation in phenotype space
  called canonical equation of adaptive dynamics. From this canonical equation and still assuming small mutation steps, we prove a
  rigorous characterization of the evolutionary branching points.
\end{abstract}

\bigskip


\bigskip

\section{Introduction}
\label{sec:intro}

Since the first works of J. Monod \cite{monod} and Novik and Szilar \cite{ns}, \cite{ns2} (see also \cite{SW}), biologists have
developed procedures which allow to maintain a bacterial population at a stationary finite size while, at the same time, the bacteria
have a positive individual growth rate. The procedure is based on a chemostat: bacteria live in a growth container of constant volume
in which liquid is injected continuously. This liquid contains nutrients which are consumed by the bacteria. We assume that the
chemostat is well stirred, so that the distribution of bacteria and nutrients are spatially uniform. Moreover, since the container
has a finite volume and fresh liquid enters permanently, an equal amount of liquid pours out containing both unconsumed nutrients and
bacteria. This pouring out helps regulating the bacteria population size. These chemostats are extremely useful to make bacteria
adapt to the nutrients environment in such a way that they will increase their growth rate. This experimental device is for example
used (on large scales) in water treatment stations. In this paper we study a chemostat model where different nutrients arrive
continuously and are simultaneously consumed by bacteria which reproduce and die in a stochastic time scale. Bacteria are
characterized by genetic parameters which are inherited during reproduction except when a mutation occurs. These parameters as well
as the concentrations of each resource, influence the demographics of the bacteria population.

Usually, chemostat models are essentially deterministic where both nutrients and bacteria population dynamics are described by
coupled deterministic continuous process. This point of view is based on the fact that reproduction of bacteria happens at the same
time scale than nutrient dilution. In this work, we consider that the bacteria population is not so large that deterministic
approximation of its size can be justified. We develop stochastic chemostat models based on the previous work of Crump and O' Young
\cite{CW} (see also~\cite{CJL}), where bacteria dynamics is modeled by a birth and death process whose demographic parameters depend
on the concentration of a unique resource, the later evolving continuously in time. This model has been rigorously studied in
Collet-M\'el\'eard-Martinez \cite{CMM12}

The literature on multi-resource chemostats is extensive but mostly specialized to very precise models. In a more general framework as we consider here, we refer for example to \cite{SW}, \cite{CJR10}. 

Our goal is to study the Darwinian evolution of the genetic parameters of the bacteria in the chemostat. We will study how mutations
and competition for nutrients lead to their progressive adaptation. The study of adaptive dynamics has been developed in the last two
decades, in particular by Metz, Geritz and coauthors~\cite{MNG92,Mal96}, and Dieckmann and Law~\cite{DL96}. This theory emphasizes
the connections between ecology and evolution: competition for resources strongly influences the selection of advantaged bacteria.
The theory was put on a rigorous mathematical footing by Champagnat and M\'el\'eard and coworkers~\cite{C06,CFM07,CM11} with a
probabilistic approach (see also~\cite{diekmann-jabin-al-05,barles-perthame-06,CJ11} for a PDE approach). These results are based on
the combination of large population and rare mutation scalings, allowing to describe evolution as a succession of mutant invasions.
In all these works, all the models are either Lotka-Volterra models or models of competition for resources assumed to be at a
quasi-stable equilibrium. This amounts in all cases to assume direct competition between individuals.

Our model is more realistic and takes explicit dynamics for several resources into account. This leads to competitive interactions
between bacteria, driving selection, with a more complicated nonlinearity than the Lotka-Volterra direct competition (although links
exist under specific parameter scalings, see \cite{Mirrahimi2}). The modeling of adaptation for multi-resources chemostat has never
been studied from a probabilistic point of view and only recently with a PDE's approach (see \cite{Mirrahimi1}). Our goal is to study
adaptive dynamics in these models using the approach of~\cite{C06,CM11} combining large populations, rare and small mutations. Let us
emphasize that adding resources dynamics makes the mathematical analysis more complicated and difficult than for the Lotka-Volterra
model.

We first introduce the model in Section~\ref{sec:model}. We construct a stochastic multi-resource chemostat model which couples
deterministic and stochastic dynamics. The bacteria dynamics follows a birth and death process with reproduction due to resources
consumption and mutation. The resource process is deterministic between birth or death events in the bacteria population, with
coefficients depending on the composition of the population. Then we introduce multi-resources chemostat deterministic systems as
large size approximations of these probabilistic models.

The core ingredient in our proofs is a stability result on deterministic chemostat systems, seeing the stochastic process as a sort
of perturbation. In Section~\ref{sec:asympt-chemostat}, we thus study the long time behavior of such deterministic nonlinear systems
and prove the convergence to a unique equilibrium as soon as some non-degeneracy assumptions hold true. The proof is based on several
Lyapunov functionals which considerably extend the usual Lyapunov functional for chemostat systems (see \cite{SW}).

In Section~\ref{sec:GD}, we study the long time stability of the stochastic process, viewed as an approximation of the deterministic
chemostat system. We prove that all traits with zero density at equilibrium actually go extinct after a time of the order the
logarithm of the population size. We also prove that the time of exit from a neighborhood of the equilibrium grows exponentially in
the population size. This kind of result follows classically from large deviation estimates, which we prove here in the non standard
situation of perturbed resources dynamics, using the Lyapunov functionals of Section~\ref{sec:asympt-chemostat}.

We finally study in Section~\ref{sec:PES} the individual-based process on the evolutionary time scale in three steps: First in
Subsection~\ref{sec:cv-PES} we consider a large population and rare mutation scaling but we do not let the size of each mutation
go to $0$. This ensures that the stochastic model on the mutation time scale (evolutionary time scale) converges to a pure jump
process describing the successive invasions of advantageous mutants. This process, called Polymorphic Evolution Sequence (PES),
generalizes the TSS introduced in \cite{Mal96} and the PES for Lotka-Volterra models whose existence has been proved in \cite{CM11}.
The difficulty consists in extending these results to our more complicated model, which is done in Appendix~\ref{sec:pf-PES}. Second,
we introduce a scaling of small mutations in the PES, letting the size of each mutation go to $0$. We prove in
Section~\ref{sec:small-mut} that, in this limit, the dynamics of co-existing traits is governed by a system of ODEs extending the
canonical equation of adaptive dynamics (see \cite{DL96,CM11}). Finally, from this canonical equation and still assuming small
mutation steps, we prove a rigorous characterization of the evolutionary branching points. The only result which needs a different
approach from the Lotka-Volterra case~\cite{CM11} is the branching criterion, where we use the results of
Section~\ref{sec:asympt-chemostat} to prove that coexistence is maintained after evolutionary branching and that the distance between
the two branches increases. The details of the proof are given in Appendix~\ref{app:br}, after giving useful results on the sign of
the fitness function (Appendix~\ref{sec:sign-fitness}), which governs the possibility of invasion of a mutant trait in a resident
population at equilibrium.

\section{The Model}
\label{sec:model}

\subsection{The stochastic model}
\label{sec:IBM}

We consider an asexual population and a hereditary phenotypic trait. Each individual $i$ is characterized by its phenotypic trait
value $x^i\in {\cal X}$, hereafter referred to as its phenotype, or simply trait. The trait space ${\cal X}$ is assumed to be a
compact subset of $\mathbb{R}^{\ell}$. The individual-based microscopic model from which we start is a stochastic birth and death
process, with density-dependence through a reproduction depending on the resources in the chemostat. There are $r$ different
resources. Resources are injected in the fluid at the constant rate $1$. Their concentrations decrease first because of the linear
pouring out of the fluid from the chemostat and second by their consumption by the bacteria. We assume that the population's size
scales with an integer parameter $K$ tending to infinity while the effect of the individual resources consumption scales with
$\frac{1}{K}$. This allows taking limits in which individuals are weighted with $\frac{1}{K}$. In Section~\ref{sec:small-mut},
another crucial scale will be the mutation amplitude $\sigma$.

\me We consider, at any time $t\geq 0$, a finite number $N^K_t$ of individuals, each of them holding trait values in
${\cal X}$. Let us denote by $(x_{1},\ldots,x_{N^K_t})$ the trait values of these individuals. The state of the population at time
$t\geq 0$, rescaled by $K$, is described by the finite point measure on ${\cal X}$
\begin{equation}
  \label{eq:nu_t}
  \nu^{K}_t=\frac{1}{K}\sum_{{i=1}}^{N^K_{t}} \delta_{x_{i}}, 
\end{equation}
where $\delta_{x}$ is the Dirac measure at $x$. This measure belongs to the set of finite point measures on $\mathcal{X}$ with mass
$1/K$
$$
\mathcal{M}^K=\left\{\frac{1}{K}\sum_{i=1}^n\delta_{x_i};\ n\geq 0,\ x_1,\ldots,x_n\in\mathcal{X}\right\}.
$$
Let $\: \langle\nu,f\rangle$ denote the integral of the measurable function $f$ with
respect to the measure $\nu$ and $\mathrm{Supp}(\nu)$ denotes its support. Then $\: \langle\nu^{K}_t,{\bf 1}\rangle=\frac{N^K_t}{K}$
and for any $x\in {\cal X}$, the positive number $\langle\nu^{K}_t,{\bf 1}_{\{x\}}\rangle$ is called {the density} at time $t$ of
trait $x$.

\me The concentrations of the $r$ resources at time $t$ are described by a $r$-dimensional vector
\begin{equation*}
  \mathbf{R}^K(t) =(R^K_{1}(t), \cdots, R^K_{r}(t)) \in \mathbb{R}_{+}^r.
\end{equation*}

\me We introduce the following demographic parameters.

\begin{itemize}
\item An individual with trait $x$ reproduces with birth rate given by
  \begin{equation}
    \label{birth}
    \sum_{{k=1}}^r \eta_{k}(x) \, R^K_{k},
  \end{equation}
  where $\eta_{k}(x)>0$ represents the ability of a bacteria with trait $x$ to use resource $R_{k}^K$ for its reproduction.
\item Reproduction produces a single offspring. With probability $1- \mu_K\, p(x)$, the newborn holds the trait value $x$ of the
  parent.
\item A mutation occurs with probability $\mu_K\, p(x)$ with $p(x)>0$ and affect the trait of the descendant involved in the
  reproduction event. The trait of the descendant is $x+h$ with $h$ chosen according to $m_{\sigma}(x,h)dh$, where $\sigma\in(0,1]$
  scales the mutation amplitude.

  \noindent For any $x$ the probability measure scales as $m_{\sigma}(x,h)dh = \frac{1}{\sigma^{\ell}} m(x,\frac{h}{\sigma})dh$,
  where $m(x,h)dh$ is a probability measure with support $\,\{h, x+h\in{\cal X}\}$ and such that $x+\sigma h\in{\cal X}$ for all $h$
  in the support of $m(x,\cdot)$ and for all $\sigma\in(0,1)$ (this last point is for example true when ${\cal X}$ is a convex subset
  of $\mathbb{R}^{\ell}$). In other words, the support of $m_{\sigma}$ has a diameter smaller than $\sigma\text{diam}({\cal X}) $.
  
\item Small $\mu_K$ means rare mutations and we assume in the sequel that 
\begin{equation}
  \label{echmut}
  \lim_{K\to \infty} K\, \mu_{K} = 0.
\end{equation}

\item Each individual with trait $x$ disappears from the chemostat at rate $d(x)\in(0,+\infty)$ either from natural death or pouring
  out with the liquid. Since the fluid pours out of the chemostat at rate 1, one could assume $d(x)\geq 1$, although this is not
  necessary for our study.
\item The concentrations of resources in the liquid are solutions of the piecewise deterministic equations: for any $k\in \{1,\cdots,
  r\}$,
  \begin{equation}
    \label{eq:EDO-R}
  \frac{d R^K_{k}(t)}{dt} = g_{k} - R^K_{k}- R^K_{k}\left(\frac{1}{K}\sum_{i=1}^{N^K(t)}\eta_{k}(x_{i})\right) = g_{k} - R^K_{k} -
  R^K_{k}\, \langle \nu^{K}_t, \eta_{k}\rangle,    
  \end{equation}
  where $g_k>0$ represents the $k$-th resource injection in the chemostat. The term $- R^K_{k}$ represents the pouring out of the
  $k$-th resource and the term $ - R^K_{k}\, \langle \nu^{K}, \eta_{k}\rangle$ describes the resource consumption by bacteria. Note
  that any solution to such an equation satisfies for any $k\in \{1,\cdots, r\}$ the inequality $R^K_{k}(t) \leq R^K_{k}(0)\wedge g_{k}$.
\end{itemize}

\bigskip
\noindent The process $((\nu^{K}_t,\mathbf{R}^K(t)),t\geq 0)$ is a ${\cal M}^K\times\RR_+^r$-valued Markov process with infinitesimal
generator defined for any bounded measurable functions $\phi$ from ${\cal M}^K\times (\mathbb{R}_{+})^r$ to $\mathbb{R}$ and $\nu
=\frac{1}{K}\sum_{i=1}^n \delta_{x_i} $ and $\ R=(R_{1},\cdots, R_{r})$ by

\begin{align}
  L^K\phi(\nu, R) & = \int_{\cal X}\bigg\{\left(\phi\left(\nu+\frac{\delta_x}{K}, R\right)
    -\phi(\nu, R)\right)(1-\mu_K\, p(x))\left(\sum_{k=1}^r \eta_{k}(x)\, R_{k}\right)  \notag \\
  & +\int_{\mathbb{R}^{\ell}} \left(\phi\left(\nu+\frac{\delta_{x+h}}{K}, R\right)-\phi(\nu, R)\right)
  \mu_K\, p(x)\left(\sum_{k=1}^r \eta_{k}(x)\, R_{k}\right)\, m_{\sigma}(x,h)dh \notag \\
  & +\left(\phi\left(\nu-\frac{\delta_x}{K},R
    \right)-\phi(\nu,R)\right)\, d(x) \bigg\}\, K\nu(dx)\notag\\
  & +\sum_{k=1}^r \frac{\partial \phi}{\partial R_{k}}(\nu, R) \left(g_{k} - R_{k} - R_{k}\,\langle \nu, \eta_{k}\rangle\right).
  \label{eq:generator-renormalized-IPS}
\end{align}
The first term describes the births without mutation, the second term the births with mutation, the third term the deaths and the
fourth term the dynamics of resources.
  
  \subsection{An example}
\label{sec:example-1}

We consider a case with two resources and one-dimensional traits having opposite effects on the two resources consumption. More
precisely let us define the following parameters:
 
\me ${\cal X}=[-1,1]$, $ \mu_K\, p(x)\equiv p$, $m(x,h)dh={\cal N}(0,\sigma^2)$ (conditioned on $x+h\in{\cal X}$), $r=2$ (2
resources), $g_1=g_2=1$, $d(x)=\frac{1}{2}+a x^2$ with $a>0$ (minimum at $0$), $\eta_1(x)=(x-1)^2$, $\eta_2(x)=(x+1)^2$. The
parameter $a$ measures the impact of trait on mortality. High $a$ means strong effect of traits away from 0 on mortality.
  
\me Simulations are given in Figure \ref{sim1}, where initially $K$ individuals have the same trait $- 0.5$, for three different
mortality parameters, $a=1/4$, $a=1/2$ and $a=1$. The upper panels represent the time evolution of the density of traits in the
population, and the lower panels the resources concentrations. We observe two different behaviors. In the four pictures, the support
of the population process first approaches $0$, where mortality is the smallest, but next, for $a=1/2$ and $a=1$, the support of the
population process stays close to 0 for a long time, while for the two simulations with $a=1/4$, the population divides into two
subpopulations with distinct trait values, but still interacting for the same resources. This phenomenon is known as
\emph{evolutionary branching}.

\me The two different simulations of evolutionary branching for $a=1/4$ (Fig.~\ref{sim1}~(a) and~(b)\:) illustrate two slightly
different ways of branching. In the first one, evolutionary branching occurs when the traits in the population are distributed around
a value not exactly equal to 0, contrary to the second simulation. As will appear in this paper, this explains why, in the first
picture, the two branches are of very different size at the beginning of evolutionary branching, while they are of similar size in
the second picture.

\me Fig.~\ref{sim1}~(c) does not show evolutionary branching, but the width of the trait distribution around 0 is wider than in the
fourth picture, where $a$ is bigger. As will appear below, the third picture actually corresponds to a critical case with respect to
evolutionary branching. Finally, in Fig.~\ref{sim1}~(d), we also observe an interesting situation in the first phase of evolution,
where the trait distribution in the populations is approaching to 0. There is a jump in the support of the trait distribution around
time $t=300$, which can be understood as follows: at this time, a mutant\footnote{Actually two different mutants traits, if one looks
  closely at the simulation.} just appeared by chance, with a much bigger trait than the largest trait in the population just before.
This mutant survived and produced a large descendence with a significantly smaller death rate than the rest of the population, which
was therefore competitively disadvantaged, and went extinct fast due to this competitive pressure.

\begin{figure}
\label{sim1} 
 \begin{minipage}[b]{0.47\textwidth}
   \begin{center}     
    \includegraphics[width=6cm]{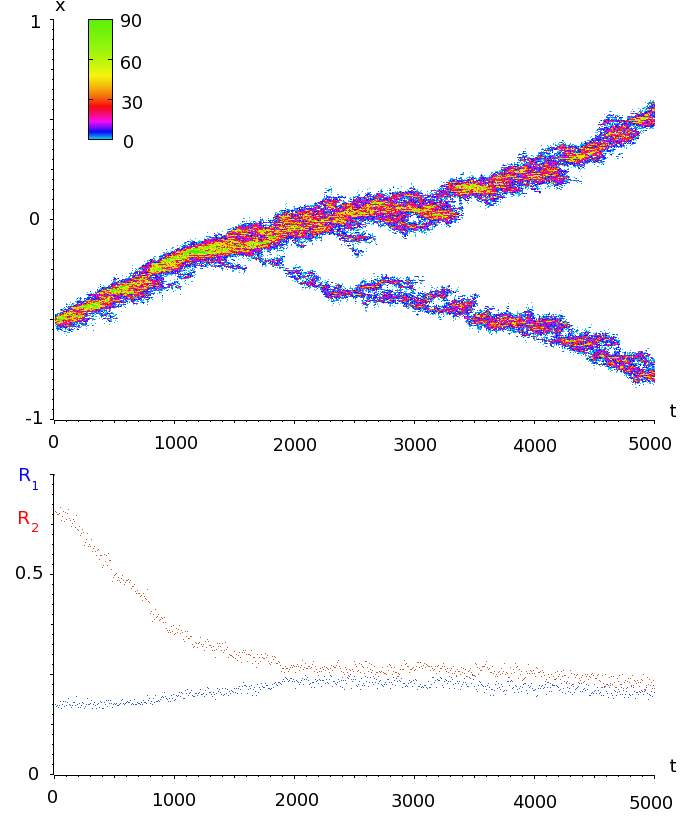}
    \medskip\\
    (a) $K=300,\ p=0.1,\ \sigma=0.01,\ a=1/4$
   \end{center}
  \end{minipage}
  \begin{minipage}[b]{0.47\textwidth}
   \begin{center}     
    \includegraphics[width=6cm]{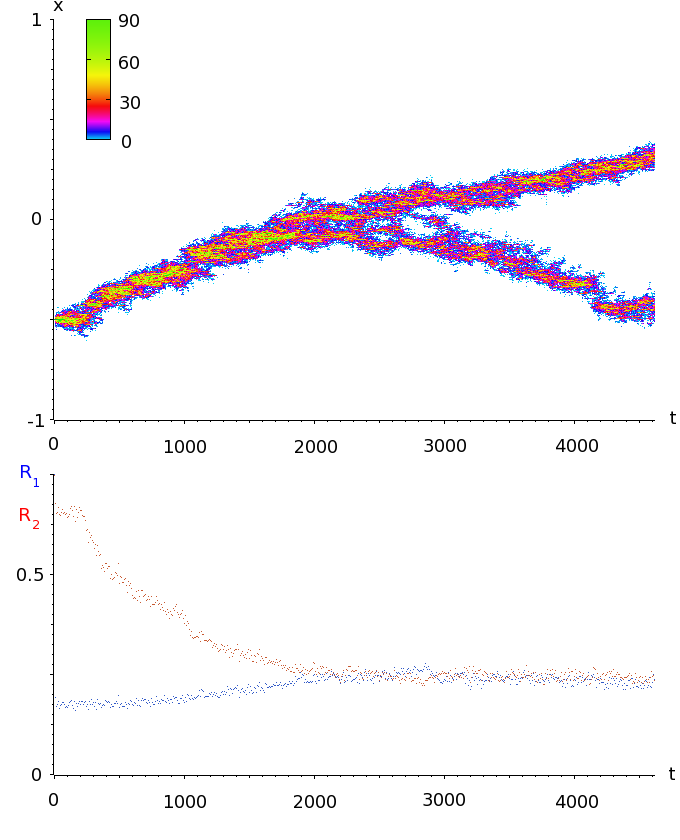}
    \medskip\\
    (b) $K=300,\ p=0.1,\ \sigma=0.01,\ a=1/4$
   \end{center}
  \end{minipage}
\bigskip

 \begin{minipage}[b]{0.47\textwidth}
   \begin{center}     
    \includegraphics[width=6cm]{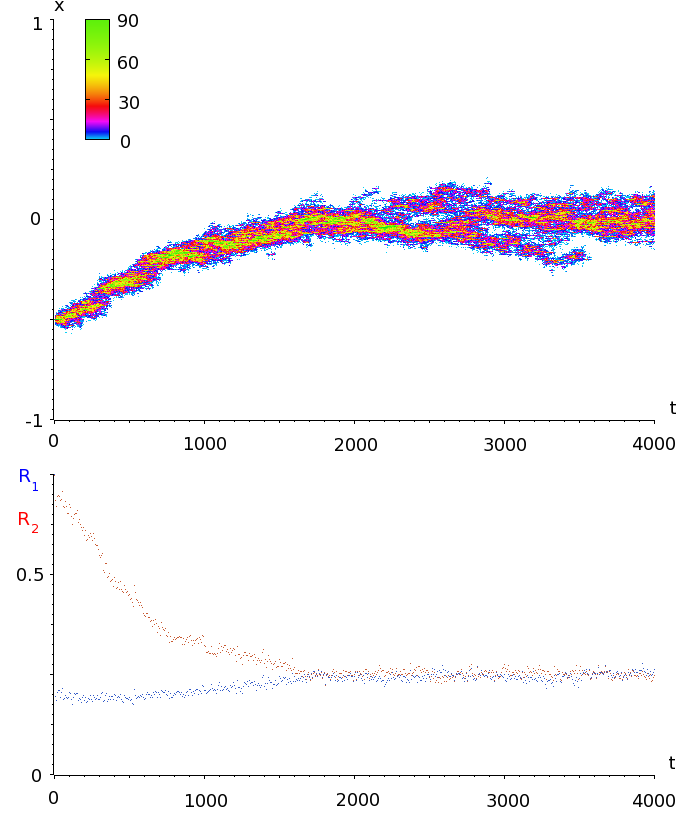}
    \medskip\\
    (c) $K=300,\ p=0.1,\ \sigma=0.01,\ a=1/2$
   \end{center}
  \end{minipage}
  \begin{minipage}[b]{0.47\textwidth}
   \begin{center}     
    \includegraphics[width=6cm]{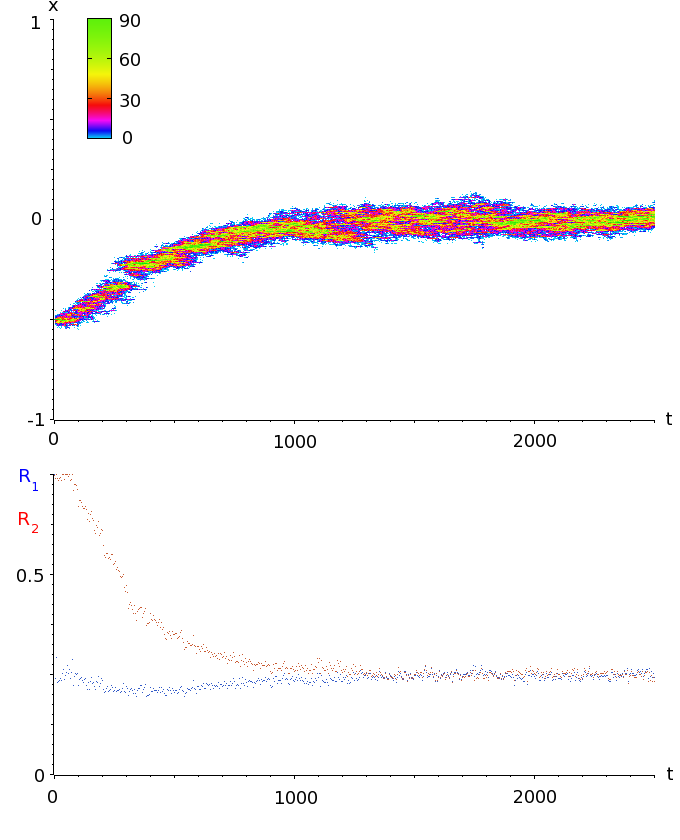}
    \medskip\\
    (d) $K=300,\ p=0.1,\ \sigma=0.01,\ a=1$
   \end{center}
  \end{minipage}\\
  \caption{Simulations of the individual-based model for three different values of the parameter $a$. Upper panels: time evolution of
    the trait density in the population. Lower panels: time evolution of the resources concentrations.}
\end{figure}

\me Our goal in this paper is to give a description of these pictures under a specific scaling of the parameters of the
individual-based model, and to give some conditions under which evolutionary branching appears. The evolution of the population
results from an instantaneous trade-off between death rate minimization and the birth rate maximization (through a better consumption
of resources) of the individuals.

\subsection{The process at the ecological time scale}
The ecological time scale is the birth and death time scale in which  the process is defined. The evolution  takes place on the longer time scale of mutations.

\bi The following properties will be assumed in the sequel.
\begin{align}
  &
  \begin{aligned}
    & \text{\it The functions\ }\eta_{1},\ldots,\eta_r \text{\it\ and\ } d \text{\it\ are\ } C^2 \text{\it\ on\ } {\cal X} \text{\it\ and\ } p \text{\it\ is Lipschitz continuous on\ } {\cal
      X}.\\
   & \text{\it Since\ } {\cal X} \text{\it\ is compact, these functions are lower bounded by positive constants.}
  \end{aligned}
  \label{A1} \\
  & \notag \\ &
  \begin{aligned}
    & \text{\it The function\ } m(x,h) (\text{\it and thus\ } m_\sigma(x,h)) \text{\it\ is Lipschitz continuous on\ } {\cal
      X}\times\mathbb{R}^{\ell}.\\
    & \text{\it In addition, for any\ } \sigma\in(0,1], \hbox{\it\ there exists a function\ }
    \bar{m}_{\sigma}:\mathbb{R}^{\ell}\rightarrow\mathbb{R}_+ \text{\it\ such
      that\ }\\ &  m_{\sigma}(x,h)\leq \bar{m}_{\sigma}(h) \text{\it\ for any\ } x\in{\cal X} \text{\it\ and\ } h\in\mathbb{R}^{\ell}
    \text{\it\ and\ } \int_{\mathbb{R}^\ell}
    \bar{m}_{\sigma}(h)dh<+\infty.    
  \end{aligned}
  \label{A2}
   \\
  & \notag \\ & \text{\it The initial conditions satisfy} \quad 
\sup_K\mathbb{E}(\langle\nu^K_0,1\rangle)<\infty\quad ;\quad \sup_{K, k\in\{1,\ldots,r\}}\,R^K_k(0)<\infty.
   \label{A2bis} 
\end{align}

\me
For fixed $K$, under~\eqref{A1}--\eqref{A2}--\eqref{A2bis}, the
existence and uniqueness in law of a process on $\mathbb{D}(\mathbb{R}_+, {\cal M}^K\times (\mathbb{R}_{+})^r)$ with infinitesimal
generator $L^K$ can be adapted from the one in Fournier-M\'el\'eard \cite{FM04} or \cite{CFM07}. 
  The process is constructed as
solution of a stochastic differential equation driven by point Poisson measures describing each jump event plus a drift term
describing the resource dynamics and Assumptions~\eqref{A1} and~\eqref{A2} prevent the population
from exploding since resources concentrations are bounded.

 \me Let us recall the construction. Let $N_1(ds,di,d\theta)$, $N_2(ds,di,d\theta,dh)$ and $N_3(ds,di,d\theta)$ be independent Poisson point measures on
$\mathbb{R}_+\times\mathbb{N}\times\mathbb{R}_+$, $\mathbb{R}_+\times\mathbb{N}\times\mathbb{R}_+\times\mathbb{R}^{\ell}$ and
$\mathbb{R}_+\times\mathbb{N}\times\mathbb{R}_+$ respectively, and with intensity measures $q_1(ds,di,d\theta)$,
$q_2(ds,di,d\theta,dh)$ and $q_3(ds,di,d\theta)$ respectively, where
$$
q_1(ds,di,d\theta)=q_3(ds,di,\theta)=\mathbbm{1}_{\{s\geq 0, \theta\geq 0\}}ds\left(\sum_{k=1}^\infty \delta_k(di)\right)d\theta
$$
and
$$
q_2(ds,di,d\theta,dh)=\mathbbm{1}_{\{s\geq 0, \theta\geq 0\}}ds\left(\sum_{k=1}^\infty \delta_k(di)\right)d\theta\bar{m}_{\sigma}(h)dh.
$$

\noindent  For all $\nu\in{\cal M}^K$, we define $x_i(\nu)\in{\cal X}$ for all $i\geq 1$ such that
$\
\nu=\frac{1}{K}\left(\sum_{i\geq 1}\delta_{x_i(\nu)}\right)
$
and
$$
x_1(\nu)\preceq x_2(\nu)\preceq\ldots\preceq x_{n(\nu)}(\nu),
$$
where $n(\nu):=K\langle\nu,1\rangle$ and $\preceq$ is an arbitrary total order on ${\cal X}\subset\mathbb{R}^{\ell}$ (e.g.\ the lexicographic order.

\noindent Now, let us consider the equation
\begin{align}
  \nu^K_t & =\nu_0^K+\frac{1}{K}\left\{ \int_0^t\int_{\mathbb{N}}\int_0^\infty\delta_{x_i(\nu^K_{s-})}\mathbbm{1}_{\{i\leq n(\nu^K_{s-}),\
      \theta\leq(1-\mu_Kp(x_i(\nu^K_{s-})))\sum_{k=1}^r\eta_k(x_i(\nu^K_{s-}))R_k(s)\}}N_1(ds,di,d\theta) \right. \notag \\ &
  +\int_0^t\int_{\mathbb{N}}\int_0^\infty\int_{\mathbb{R}^{\ell}}\delta_{x_i(\nu^K_{s-})+h}\mathbbm{1}_{\{i\leq n(\nu^K_{s-}),\ \theta\leq
    m(x_i(\nu^K_{s-}),h)\mu_Kp(x_i(\nu^K_{s-}))\sum_{k=1}^r\eta_k(x_i(\nu^K_{s-}))R_k(s)\}}N_2(ds,di,d\theta,dh) \notag \\ & \left.
    -\int_0^t\int_{\mathbb{N}}\int_0^\infty\delta_{x_i(\nu^K_{s-})}\mathbbm{1}_{\{i\leq n(\nu^K_{s-}),\ \theta\leq
      d(x_i(\nu^K_{s-}))\}}N_3(ds,di,d\theta)\right\}, \label{eq:constr-nu}
\end{align}
coupled with the equations~(\ref{eq:EDO-R}) for resources with
$R_k(0)\geq 0$.

\me
\noindent This system of equations admits a unique solution, up to a  time  of
accumulation of jumps.  The next proposition shows that this time is infinite and gives uniform   moment estimates. 

\begin{prop}
  \label{prop:moments-nu}
  Assume ~\eqref{A1}--\eqref{A2}--\eqref{A2bis}, then the process does not explode in finite time  almost surely  and for all $t\geq 0$,
      \begin{equation}
      \label{eq:a}
    \sup_{K\geq 1}\sup_{t\geq 0}\mathbb{E}\left[\langle\nu^K_{t},1\rangle+\sum_{k=1}^{r}R^K_k(t)\right]<\infty.
     \end{equation}
\end{prop}

\begin{proof}
Since $0\leq R^K_k(t)\leq g_k\vee \|R^K_k(0)\|_{\infty}=: R^*_k$ for all $t\geq 0$,  the
individual birth rate in the population is always smaller than $\sum_{k=1}^r \eta_k^* R_k^*$, where $\eta^*_k:= \| \eta_k\|_{\infty}$. Therefore, for all $t\leq 0$,
$\
K\langle\nu^K_t,1\rangle$
is stochastically dominated  by a pure birth (Yule) process in $\mathbb{Z}_+$ with transition rate $k\sum_{k=1}^r \eta_k^* R_k^*$ from $k$
to $k+1$ (this process can be easily explicitly constructed from the point processes $N_1$, $N_2$ and $N_3$). The Yule process is a.s. finite for any $t$, so that  the process is well defined for all times. In addition, we get that
\begin{align*}
  \langle\nu^K_t,1\rangle+\sum_{k=1}^r R^K_k(t) & =\langle\nu^K_0,1\rangle+\sum_{k=1}^r R^K_k(0)
  +\int_0^t\int_{{\cal X}}\left(\sum_{k=1}^r\eta_k(x)R^K_k(s)-d(x)\right)\nu^K_s(dx)ds \\
  & +\sum_{k=1}^r\int_0^t\left(g_k-R^K_k(s)-R^K_k(s)\langle\nu^K_s,\eta_k\rangle\right)ds+M_t,
\end{align*}
with
\begin{align*}
  M_t & =\frac{1}{K}\int_0^t\int_{\mathbb{N}}\int_0^\infty\mathbbm{1}_{\{i\leq n(\nu^K_{s-}),\
    \theta\leq(1-\mu_Kp(x_i(\nu^K_{s-})))\sum_{k=1}^r\eta_k(x_i(\nu^K_{s-}))R^K_k(s)\}}\tilde{N}_1(ds,di,d\theta) \\ &
  +\frac{1}{K}\int_0^t\int_{\mathbb{N}}\int_0^\infty\int_{\mathbb{R}^{\ell}}\mathbbm{1}_{\{i\leq n(\nu^K_{s-}),\ \theta\leq
    m(x_i(\nu^K_{s-}),h)\mu_Kp(x_i(\nu^K_{s-}))\sum_{k=1}^r\eta_k(x_i(\nu^K_{s-}))R^K_k(s)\}}\tilde{N}_2(ds,di,d\theta,dh) \\ &
  +\frac{1}{K}\int_0^t\int_{\mathbb{N}}\int_0^\infty\mathbbm{1}_{\{i\leq n(\nu^K_{s-}),\ \theta\leq d(x_i(\nu^K_{s-}))\}}\tilde N_3(ds,di,d\theta),
\end{align*}
where $\tilde{N}_i=N_i-q_i$, $i=1,2,3$, are the compensated Poisson processes. The process  $(M_t, t\geq 0)$ is a martingale and thus
$$
\mathbb{E}\left[\langle\nu^K_t,1\rangle+\sum_{k=1}^r R^K_k(t)\right] =\mathbb{E}\left[\langle\nu^K_0,1\rangle+\sum_{k=1}^r
  R^K_k(0)\right]+\int_0^t\left(-\langle\nu^K_s,d\rangle+\sum_k(g_k-R^K_k(s))\right)ds.
$$
\eqref{eq:a} then follows from the fact that $d(\cdot)\geq\underline{d}>0$ by Assumption~\eqref{A1} and Gronwall's lemma.
In addition, because of the shift invariance of the
Poisson point measures $N_1$, $N_2$ and $N_3$, the process $(\nu^K_t,\mathbf{R}^K(t))_{t\geq 0}$ is strong Markov.
\end{proof}

\bigskip

\noindent  The next result shows that mutations cannot occur on  bounded time intervals,  since    the mutation time scale  $t/K\mu_K$ tends to infinity by \eqref{echmut}. 

\me We define
  $T^K_{\text{mut}}$ as the first mutation time of the population process $(\nu^K_{t}, t\geq 0)$.

\begin{cor}
  \label{lemma:not-too-many-mutation}
  Assume ~\eqref{A1}--\eqref{A2}--\eqref{A2bis}. Then for all $\eta>0$, there exists $\varepsilon>0$ such that for all
  $t\geq 0$
  $$
  \limsup_{K\rightarrow\infty}\mathbb{P}\left(\text{a mutation occurs on\
    }\left[\frac{t}{K\mu_K},\frac{t+\varepsilon}{K\mu_K}\right]\right)\leq\eta.
  $$
\end{cor}

\begin{proof} The proof is based on a coupling inspired by the previous proposition.
  By  the Markov property, it is sufficient to prove the result for $t=0$.
 Let $\tilde{\nu}_t$ be the solution of
  \begin{align*}
    \tilde{\nu}_t & =\nu_0+\frac{1}{K}\left\{ \int_0^t\int_{\mathbb{N}}\int_0^\infty\delta_{x_i(\tilde{\nu}_{s-})}\mathbbm{1}_{\{i\leq n(\tilde{\nu}_{s-}),\
        \theta\leq(1-\mu_Kp(x_i(\tilde\nu_{s-})))\sum_{k=1}^r\eta_k(x_i(\tilde\nu_{s-}))\tilde R_k(s)\}}N_1(ds,di,d\theta) \right.  \\ &
    \left. -\int_0^t\int_{\mathbb{N}}\int_0^\infty\delta_{x_i(\tilde\nu_{s-})}\mathbbm{1}_{\{i\leq n(\tilde\nu_{s-}),\ \theta\leq
        d(x_i(\tilde\nu_{s-}))\}}N_3(ds,di,d\theta)\right\},
  \end{align*}
  coupled with the equation
  $$
  \frac{d \tilde R_{k}(t)}{dt} = g_{k} -\tilde R_{k}-\tilde R_{k}\, \langle \tilde\nu_s, \eta_{k}\rangle  
  $$
    Comparing with~(\ref{eq:constr-nu}) we see that $\tilde\nu_t=\nu_t$ and $\tilde R_k(t)=R_k(t)$ for all $t<T^K_{\text{mut}}$.

  \me Now, let us define the $\mathbb{Z}_+$-valued process
  $$
  A_t=\int_0^t\int_{\mathbb{N}}\int_0^\infty\int_{\mathbb{R}^{\ell}}\mathbbm{1}_{\{i\leq n(\tilde\nu_{s-}),\ \theta\leq
    \bar m_{\sigma}(h)\mu_K r\bar\eta\bar g\}}N_2(ds,di,d\theta,dh),
  $$
  where  $\bar\eta=\max\{\sup_{x\in\mathcal{X}}\eta_k(x); 1\leq k\leq r\}$ and $\bar
  g=\max\{\|R_k(0)\|_{L^\infty}\vee g_k; 1\leq k\leq r\}$. Then $T^K_{\text{mut}}\geq T_A$, where $T_A$ is the first jump time of
  $(A_t,t\geq 0)$. Since  $N_2$ is independent of $(\tilde\nu,\tilde{\mathbf{R}})$, we have
  \begin{align*}
    \mathbb{P}\left(T^K_{\text{mut}}>\frac{\varepsilon}{K\mu_K}\right) & \geq\mathbb{P}\left(T_A>\frac{\varepsilon}{K\mu_K}\right)  =\mathbb{E}\left[\exp\left(-\int_0^{\varepsilon/K\mu_K}n(\tilde\nu^K_s)\mu_K r\bar\eta\bar g ds\right)\right] \\
    & \geq 1-K\mu_K r\bar\eta\bar g\int_0^{\varepsilon/K\mu_K}\mathbb{E}\langle\tilde\nu^K_s,1\rangle\, ds,
  \end{align*}
  and Lemma~\ref{lemma:not-too-many-mutation} follows from the boundedness of the moments (cf. Proposition \ref{prop:moments-nu}).
\end{proof}

\subsection{Convergence to deterministic chemostat systems when $K\rightarrow+\infty$}
In this section we study the large population and rare mutation approximation of the process described above when the initial measure has the finite support $\{x_{1}, \cdots, x_{n}\}$. The limit is deterministic and continuous and the mutation events disappear.
The next result is a simple but
useful first step  to characterize the dynamics between mutation events.

\me
We introduce the following chemostat (coupled) system, denoted  by  $\text{CH}(n, x_{1}, \cdots, x_{n})$,  solved by $(u(t), R(t))\in \mathbb{R}_{+}^{n+r}$: 
\me For $i=1\cdots n$, for $k= 1 \cdots r$,
\begin{equation}
  \label{system}
  \begin{cases}&\displaystyle \dot{u}_{i} =
    u_{i}\bigg(-d(x_{i}) + \sum_{k=1}^r \eta_{k}(x_{i})\, R_{k}\bigg) ;\\ & \displaystyle
    \dot{R}_{k} = g_{k} - R_{k}\left(1+ \sum_{j=1}^n \eta_{k}(x_{j})\,
      u_{j}\right).\end{cases}
\end{equation}

\begin{thm}
  \label{largepop}
  Let $x_1,\ldots,x_n$ be distinct points in ${\cal X}$. Assume that $\nu^K_0=\sum_{i=1}^n u^{K,i}_0\delta_{x_i}$ such that
  $u^{K,i}_0\rightarrow u^i_0$ and $R^K_k(0)\rightarrow R_{k,0}$ in probability, where $u^i_0$ and $R_{k,0}$ are 
  deterministic, nonnegative numbers.  Then, for all $T>0$,
  $$
  \lim_{K\rightarrow+\infty}\PP(T^K_{\text{mut}}<T)=0,
  $$
  and  $$
  \sup_{0\leq t\leq T}\left(\sum_{i=1}^n|\langle\nu^K_t,\mathbbm{1}_{x_i}\rangle-u_i(t)|
    +\sum_{k=1}^r|R^K_k(t)-R_k(t)|\right)\rightarrow 0
  $$
  in probability as $K\rightarrow+\infty$, where $(u_1(t),\ldots,u_n(t),R_1(t),\ldots, R_r(t))$ is the solution of the chemostat
  system $\text{CH}(n,x_1,\ldots,x_n)$ with initial condition $u_i(0)=u^i_0$ and $R_k(0)=R_{k,0}$.
\end{thm}

\begin{proof}
   Since $K\mu_K\rightarrow 0$, the fact that $\lim_{K\rightarrow+\infty}\PP(T^K_{\text{mut}}<T)=0$ follows trivially from
  Corollary~\ref{lemma:not-too-many-mutation}. Therefore, in order to prove the second part of the result, it suffices to prove it for
  the population dynamics obtained by setting to zero the birth rates with mutation. In this case, the model reduces to a birth and
  death Markov chain in $\frac{1}{K}\mathbb{Z}_+$ for $(\langle\nu^K_t,\mathbbm{1}_{x_i}\rangle,1\leq i\leq n)_{t\geq 0}$, coupled
  with piecewise deterministic dynamics for the resources.

  \noindent Then the second part of the result can be proved using standard techniques from~\cite[Ch.~11]{EK86}. The
  only difficulty comes from the fact that the birth and death rates and the vector fields of the resources dynamics are only
  \emph{locally} Lipschitz functions of the state of the process. Since the limit function $(u_1(t),\ldots,u_n(t),R_1(t),\ldots,
  R_r(t))$ takes values in a compact set, the difficulty can be easily solved by regularizing the transition rates out of a
  sufficiently large compact set.
\end{proof}

\section{Asymptotic behavior of the deterministic chemostat system}
\label{sec:asympt-chemostat}

\subsection{Assumptions and statement of the results}
\label{sec:main-thm-chemostat}

In order to study the long time behavior of the system \eqref{system}, we need some additional assumptions.
\begin{align}
  &
  \begin{aligned}
    & \text{\it For any\ } x\in {\cal X},\quad 
  \sum_{k=1}^r\, \eta_{k}(x) g_{k} > d(x).
  \end{aligned}
  \label{A3} \\
  & \notag \\ &
  \begin{aligned}
    & \text{\it For any\ } n\geq 1 \text{\it\ and any distinct\ } x_{1}, \cdots, x_{n} \in {\cal X},  \text{\it\
      Equation~\eqref{eq:hypA4} has at most one\ } \\ &
    \text{\it solution\ }(u_{1}, \cdots, u_{n})\in\mathbb{R}_+^n, \text{\it\ where\ }
  \end{aligned}
  \label{A4}  \\ &
  \qquad\qquad\qquad\qquad\quad d(x_{i}) - \sum_{k=1}^r\,\frac{ \eta_{k}(x_i)  g_{k}}{1 + \sum_{j=1}^n \eta_{k}(x_{j}) u_{j}}=0\
  ,\quad   \, 1\leq i \leq n.
  \label{eq:hypA4}
\end{align}



\me Assumption \eqref{A3} means that when resources are maximal the population process is supercritical. Thus,  \eqref{A3} prevents the
individual-based model to become extinct too fast. It also ensures that the trivial equilibrium $(0, \cdots, 0, g_{1}, \cdots,
g_{n})$ of the deterministic system~\eqref{system} is unstable.

\me Since the equilibria $(\bar u, \bar R)$ of the chemostat system $\text{CH}(n,x_1,\ldots,x_n)$ are given by canceling the right
hand side of \eqref{system}, they must satisfy
$$
\bar R_{k } = \frac{g_{k}}{1 + \sum_{j=1}^n \eta_{k}(x_{j}) \bar u_{j}}
$$
and for all $i$ either $\bar u_{i}=0$ or $ d(x_{i})= \sum_{k=1}^r \eta_{k}(x_i) \bar R_{k }.$

\me Therefore Assumption  \eqref{A4} implies that, for all $I\subset\{1,2,\ldots,n\}$ there is at most one equilibrium $(\bar u,\bar
R)$ of~(\ref{system}) such that $\bar u_i=0$ for all $i\not\in I$ and $\bar u_i>0$ for all $i\in I$. In particular, we will make use
of the following consequence of~\eqref{A4}: if $(\bar u,\bar R)$ is an equilibrium of~(\ref{system}) and $v$ is a vector of
$\mathbb{R}_{+}^n$ such that $v_i=0$ implies $\bar u_i=0$ for all $1\leq i\leq n$ and for any $k$
$$
\sum_{j=1}^n \eta_{k}(x_{j})\bar u_{j}= \sum_{j=1}^n \eta_{k}(x_{j})v_{j},
$$
then $v=\bar u$.

\begin{prop}
  \label{prop:A4}
  \begin{description}
  \item[\textmd{(i)}] Assumption~\eqref{A4} is satisfied as soon as: For all distinct $ x_{1} ,\cdots,x_{r+1}\in {\cal X}$, the vectors
    \begin{equation}
      \label{indep1}
      \begin{pmatrix} \eta_{1}(x_{1})\\ \vdots \\ \eta_{1}(x_{r+1})\end{pmatrix} \ldots \begin{pmatrix} \eta_{r}(x_{1})\\ \vdots \\ \eta_{r}(x_{r+1})\end{pmatrix} , \begin{pmatrix} d(x_{1})\\ \vdots \\ d(x_{r+1})\end{pmatrix} 
    \end{equation}
    are linearly independent, and for all distinct $x_1,\ldots,x_r$, the vectors
    \begin{equation}\label{eq:cond-A4}
      \begin{pmatrix} \eta_{1}(x_{1})\\ \vdots \\ \eta_{1}(x_{r})\end{pmatrix} \ldots \begin{pmatrix} \eta_{r}(x_{1})\\ \vdots \\ \eta_{r}(x_{r})\end{pmatrix}
    \end{equation}
    are also linearly independent.
  \item[\textmd{(ii)}] Conversely, Assumption~\eqref{A4} implies that, for all distinct $x_1,\ldots,x_n$ such that~\eqref{eq:hypA4} admits
    a solution with all positive coordinates, the vectors
    \begin{equation}\label{eq:cond-A4-ii}
      \text{Vect}\left\{\begin{pmatrix} \eta_{1}(x_{1})\\ \vdots \\ \eta_{1}(x_{n})\end{pmatrix} \ldots \begin{pmatrix} \eta_{r}(x_{1})\\ \vdots \\ \eta_{r}(x_{n})\end{pmatrix}\right\}=\mathbb{R}^n
    \end{equation}
    are linearly independent. In particular,~\eqref{eq:hypA4} admits no solution in $(0,+\infty)^n$ if $n>r$.
  \end{description}
\end{prop}

\begin{proof} 
Let us first assume that $n\geq r+1$ and fix $x_1,\ldots,x_n$ distinct. In view of \eqref{indep1}, the system of equations
$$
d(x_{i}) - \sum_{k=1}^r\, \eta_{k}(x_{i})R_k=0\, ; \, 1\leq i \leq n
$$
has no solution $(R_1,\ldots,R_r)$. Hence the system~(\ref{eq:hypA4}) has no solution.

\me If $n\leq r$, consider two solutions $u$ and $u'$ of (\ref{eq:hypA4}) and define the vector $R=(R_1,\cdots,R_r)$ by
\begin{equation}
  \label{R}
  R_k=\frac{g_k}{1+\sum_i \eta_k(x_i)u_i},\qquad\forall 1\leq k\leq r,
\end{equation}
and the vector $R'$ similarly in function of $u'$. Then
\begin{align*}
  & \sum_{k}\frac{R_kR'_k}{g_k}\left(\sum_i \eta_k(x_i)(u_i-u'_i)\right)^2 \\
  & \qquad\qquad =\sum_k \sum_j(u_j-u'_j)\frac{R_kR'_k}{g_k}\eta_k(x_j)
  \sum_i \eta_k(x_i)(u_i-u'_i) \\
  & \qquad\qquad =  \sum_k\left( \frac{R'_k} {g_k} \sum_j u_j\, \eta_k(x_j) R_k - \frac{R_k} {g_k} \sum_j u'_j\,\eta_k(x_j) R'_k\right) \sum_i \eta_k(x_i)(u_i-u'_i)\\
  & \qquad\qquad =-\sum_i(u_i-u'_i)\sum_k \eta_k(x_i)(R_k-R'_k),
\end{align*}
where we have used \eqref{R} for $(R,u)$ and $(R',u')$. Since $\sum_k \eta_k(x_i)R_k=\sum_k\eta_k(x_i)R'_k=d(x_i)$ for all $i$, the
previous quantity is 0. Therefore,
$$
\sum_i \eta_k(x_i)(u_i-u'_i)=0,\qquad\forall k\in\{1,\ldots,r\}.
$$
Since $n\leq r$, the vectors
$$
\begin{pmatrix} \eta_{1}(x_{1})\\ \vdots \\ \eta_{r}(x_{1})\end{pmatrix} \ldots \begin{pmatrix} \eta_{1}(x_{n})\\ \vdots \\
  \eta_{r}(x_{n})\end{pmatrix}
$$
are linearly independent under condition~(\ref{eq:cond-A4}), which ends the proof of \eqref{A4} by implying that $u_i=u_i',\ \forall i$.

\noindent Assuming that Point~(ii) does not hold, and letting $x_1,\ldots,x_n$ be distinct traits and $u_1,\ldots,u_n$ be a solution
to~(\ref{eq:hypA4}) such that
$$
\alpha_1\eta_k(x_1)+\ldots+\alpha_n\eta_k(x_n)=0
$$
for all $1\leq k\leq r$ and for some $(\alpha_1,\ldots,\alpha_n)\not=0$. Then, for $\varepsilon$ close enough to 0, the vector
$(v_1,\ldots,v_n)$ belongs to $(0,\infty)^n$ and is another solution of~(\ref{eq:hypA4}), where
$v_i=u_i+\varepsilon\alpha_i$.
\end{proof}

\noindent The two assumptions \eqref{A3}--\eqref{A4} imply our main result on the large time behavior of the chemostat systems of ODEs.

\begin{thm}
  \label{equilibrium}
  Assume \eqref{A3} and \eqref{A4}. For all $n\geq 1$ and all distinct $x_{1}, \cdots, x_{n}\in {\cal X}$, there exists $(\bar{u},\bar{R})$ in
  $(\mathbb{R}_{+})^{n+r}$ such that any solution $(u(t), R(t))$ of the system \eqref{system} with $u_{i}(0) >0$ for any $1\leq i\leq
  n$, converges to $(\bar{u},\bar{R})$. In addition, $(\bar{u},\bar{R})$ is the unique equilibrium of the system \eqref{system}
  satisfying, for all $i$ such that $\bar{u}_{i} = 0$, the (weak) stability condition
  \begin{equation}
    \label{stable}
    - d(x_{i}) + \sum_{k=1}^r \eta_{k}(x_{i})\, \bar R_{k} \leq 0.
  \end{equation}
\end{thm}

\noindent Before giving the proof, we define two notions of great importance in the sequel. Writing $\mathbf{x}=(x_1,\ldots,x_n)$, we
denote by
$$
\mathbf{\bar u}(\mathbf{x})=(\bar u_1(x_1,\ldots,x_n),\ldots,\bar u_n(x_1,\ldots,x_n))
$$
the vector of equilibrium densities of the previous theorem and by 
$$
\mathbf{\bar R}(\mathbf{x})=(\bar R_1(x_1,\ldots,x_n),\ldots, \bar R_r(x_1,\ldots,x_n))
$$
the corresponding equilibrium resources concentrations.

\begin{defi}
  \label{def:coex-fitn}
  We say that the traits $x_{1}, \cdots, x_{n}$ coexist if the quantities $\bar{u}_{1}(\mathbf{x}),\ldots,\bar{u}_n(\mathbf{x})$ are
  all positive, where $\mathbf{x}=(x_1,\ldots,x_n)$. For all $n\geq 1$, we denote by ${\cal D}_n$  the domain of coexistence of $n$
  traits:
  $$
  {\cal D}_n=\{(x_1,\ldots,x_n)\in{\cal X}^n:x_1,\ldots,x_n\text{\ coexist}\}.
  $$
  Note that ${\cal D}_1={\cal X}$.\\
  For distinct traits $x_1,\ldots,x_n$, we also define the invasion fitness of a new trait $y$ as the function
  \begin{equation}
    \label{fitness}
    f(y; x_{1}, \cdots, x_{n})= -d(y) + \sum_{k=1}^r \eta_{k}(y)\, \bar
    R_{k}.
  \end{equation}
\end{defi}

\begin{rem}
  \label{rem:comment-A4}
  We can interpret Assumption~\eqref{A4} and Proposition~\ref{prop:A4} with this vocabulary:~\eqref{A4} implies that there
  exists at most a single vector of population densities where $x_1,\ldots,x_n$ coexist for all distinct $x_1,\ldots,x_n\in{\cal X}$,
  and Proposition~\textup{\ref{prop:A4}~(ii)} means that when $x_1,\ldots,x_n$ coexist, we must have $n\leq r$. In particular, ${\cal D}_n=\emptyset$ if $n>r$.
\end{rem}
If the trait $y$ is interpreted as a mutant trait trying to invade the resident populations of
coexisting traits $x_1,\ldots,x_n$, the terminology of fitness refers to the possibility of
invasion of the mutant trait. Indeed, it follows from Theorem~\ref{equilibrium} that $\bar
u(x_1,\ldots,x_n,y)=(\bar u_1(\mathbf{x}),\ldots,\bar u_n(\mathbf{x}),0)$ iff
$f(y;x_1,\ldots,x_n)\leq 0$, and that if $f(y;x_1,\ldots,x_n)>0$, the last equilibrium is
locally unstable, since the Jacobian matrix of the system at this point obviously admits
$f(y;x_1,\ldots,x_n)$ as eigenvalue.

\subsection{Proof of Theorem~\ref{equilibrium}}
\label{sec:pf-CH}

The proof of this results was sketched in~\cite{CJR10}. We shall give here a complete and detailed proof. The idea is to construct a
Lyapunov functional for the system~(\ref{system}) in three steps.

\paragraph{Step 1. Lyapunov functional and stable equilibrium for a reduced system:}

We consider the quasi-stable approximation of the system~(\ref{system}) obtained by putting at each time $t$ the resources at the
equilibrium associated with the population densities at time $t$.
\begin{equation}
  \label{eq:system2}
  \dot{u}_{i}=u_{i}\left(-d(x_{i}) + \sum_{k=1}^r
    \frac{\eta_{k}(x_{i})\, g_k}{1+\sum_{j=1}^n\eta_k(x_j)u_j}\right),\qquad\forall i\in\{1,\ldots,n\}.
\end{equation}
This system admits the Lyapunov functional 
\begin{equation}
  \label{Lyapounov}
  F(u_1,\ldots,u_n)=\sum_i d(x_i)u_i-\sum_k g_k\log\left(1+\sum_j \eta_k(x_j)u_j\right).
\end{equation}
One easily checks that
$$
\frac{d}{dt}F(u(t))=-\sum_i u_i(t)\left(-d(x_i)+\sum_k\frac{g_k\, \eta_k(x_i)}{1+\sum_j\eta_k(x_j)u_j(t)}\right)^2.
$$
This is a strict Lyapunov functional since $dF(u(t))/dt\leq 0$, and $dF(u(t))/dt=0$ if and only if $u(t)$ is an equilibrium of the
system~(\ref{eq:system2}). Clearly,
$$
F(u)\rightarrow+\infty\quad\text{when}\quad |u|\rightarrow+\infty\text{\ with\
}u\in\mathbb{R}_+^n,
$$
since $d(x)>0$ for all $x\in{\cal X}$. Therefore, for any initial condition $u(0)$, the solution $u(t)$ of~\eqref{eq:system2}
converges to an equilibrium when $t\rightarrow+\infty$. In addition, $F$ is a convex function as
\[
\frac{\partial^2 F}{\partial u_i\partial u_j}=\sum_k \frac{g_k}{\left(1+\sum_l \eta_k (x_l)\,u_l\right)^2}\;\eta_k(x_i)\,\eta_k(x_j).
\]
Hence
\begin{equation}
  \label{eq:convexite-F}
  \sum_{ij} \frac{\partial^2 F}{\partial u_i\partial u_j}\;\xi_i\,\xi_j=
  \sum_k \frac{g_k}{\left(1+\sum_l \eta_k (x_l)\,u_l\right)^2}\;\left(\sum_i \eta_k(x_i)\,\xi_i\right)^2\geq 0,
\end{equation}
thus the convexity. In general $F$ is not strictly convex as~(\ref{eq:convexite-F}) can vanish for nonzero $\xi$ if $n>r$. However we
can still prove that there is a unique local minimum. Indeed, consider any critical point $\mathbf{u}$ of $F$. Define $I$ the set of
indices $i$ s.t. $u_i>0$. Then for any $i\in I$ one has that
\[
d(x_i)-\sum_k \frac{g_k\, \eta_k(x_i)}{1+\sum_j \eta_k(x_j)\,u_j}=0
\] 
We denote $\RR^I$ the subspace of $\RR^n$ defined by $\xi\in \RR^I$ iff $\xi_i=0$ for any $i\not\in I$. At the point $\mathbf{u}$,
$F$ is strictly convex in $\RR^I$. Indeed if not, one could find $\xi\in \RR^I$ s.t. $\sum_{i} \eta_k(x_i)\xi_i=0$ for any $k$.

\noindent On the other hand for $\varepsilon$ small enough $\mathbf{u}+\varepsilon \xi$ also belongs to $\RR^n_+$ and one also has
that
\[
d(x_i)-\sum_k \frac{g_k\, \eta_k(x_i)}{1+\sum_j \eta_k(x_j)\,(u_j+\varepsilon \xi_j)}=0,
\]
which would violate the Assumptions \eqref{A1} and\eqref{A4}. This shows that if $\mathbf{u}$ is a local minimum of $F$ then no other
local minima may exist on $\RR^I$, as $F$ is convex over $\RR^I$ and strictly convex at $\mathbf{u}$.
 
\noindent Since $F(u)\rightarrow+\infty$ when $|u|\rightarrow+\infty$ with $u\in\mathbb{R}_+^n$, $F$ has at least one global minimum.
Choose a global minimum $\mathbf{u}$ such that $|I|$ is the largest and assume that there exists another one $\mathbf{u'}$, defining
another set $I'$. Since $F$ is strictly convex in $\mathbb{R}^I$, one necessarily has that $I'\not\subset I$. In addition by the
convexity of $F$ any $\mathbf{u^\theta}=\theta\mathbf{u}+(1-\theta)\mathbf{u'}$ is also a global minimum. However for $\theta\in (0,\
1)$, one has that $u^\theta_i>0$ for any $i\in I\cup I'$ which is strictly larger than $I$. Hence we obtain a contradiction and $F$
has a unique global minimum.

\me Therefore, $F$ admits a unique global minimizer in the closed, convex set $\mathbb{R}_+^n$, denoted $\bar u$. Let us denote by
$\bar R$ the vector with coordinates
$$
\bar R_k=\frac{g_k}{1+\sum_{i=1}^n\eta_k(x_i)\bar u_i},\qquad 1\leq k\leq r.
$$
Then, the vector $(\bar u,\bar R)$ is an equilibrium of~(\ref{system}). Since $\bar u$ is a global minimum of $F$ on
$\mathbb{R}_+^n$, for any $i$ such that $\bar u_i=0$, one must have $\frac{\partial F}{\partial u_i}\geq 0$. This
yields~(\ref{stable}).

\me Let us check that $(\bar u,\bar R)$ is the only equilibrium of~(\ref{system}) satisfying this property. This is equivalent to
checking that $\bar u$ is the only equilibrium of~(\ref{eq:system2}) such that
$$
- d(x_{i}) + \sum_{k=1}^r \frac{g_k\, \eta_{k}(x_{i})}{1+\sum_j\eta_k(x_j)\bar u_j} \leq 0
$$
for all $i$ such that $\bar u_i=0$. Since $\bar u$ is an equilibrium of~(\ref{eq:system2}), note also that
$$
- d(x_{i}) + \sum_{k=1}^r \frac{g_k\,\eta_{k}(x_{i})}{1+\sum_j\eta_k(x_j)\bar u_j} = 0
$$
for all $i$ such that $\bar u_i>0$.

\me Let us consider two such equilibria, $\bar u^1\not=\bar u^2$. Then, adapting the computation for the convexity of $F$
\begin{align*}
  0 & \geq\sum_i \bar u^1_i\left(-d(x_i)+\sum_k\frac{g_k\, \eta_k(x_i)}{1+\sum_j\eta_k(x_j)\bar
      u^2_i}\right) \\
  & \qquad+\sum_i \bar u^2_i\left(-d(x_i)+\sum_k\frac{g_k\, \eta_k(x_i)}{1+\sum_j\eta_k(x_j)\bar
      u^1_i}\right) \\
  & =\sum_i (\bar u^1_i-\bar
  u^2_i)\left(-d(x_i)+\sum_k\frac{g_k\, \eta_k(x_i)}{1+\sum_j\eta_k(x_j)\bar u^2_i}\right) \\
  & \qquad+\sum_i (\bar u^2_i-\bar
  u^1_i)\left(-d(x_i)+\sum_k\frac{g_k\,\eta_k(x_i)}{1+\sum_j\eta_k(x_j)\bar u^1_i}\right) \\
  & = \sum_k\frac{g_k}{(1+\sum_j \eta_k(x_j)\bar u^1_j)(1+\sum_j \eta_k(x_j)\bar u^2_j)}
  \left(\sum_i \eta_k(x_i)(\bar u^2_i-\bar u^1_i)\right)^2.
\end{align*}
This cannot hold unless
$$
\sum_i \eta_k(x_i)(\bar u^2_i-\bar u^1_i)=0,\qquad\forall 1\leq k\leq r.
$$
But this would imply that $\bar R^1=\bar R^2$ (with obvious notations). Defining $J:=\{i:\sum_k\eta_k(x_i)\bar R^1_k=d(x_i)\}$, since
$\bar u^1$ and $\bar u^2$ are equilibria of~(\ref{eq:system2}), we would then have $\bar u^1_i=\bar u^2_i=0$ for all $i\not\in J$,
and we would obtain a contradiction with Assumption~\eqref{A4} applied to $(x_i)_{i\in J}$. Hence $(\bar u,\bar R)$ is the only equilibrium
of~(\ref{system}) satisfying~(\ref{stable}).

\paragraph{Step 2. A degenerate Lyapunov functional for the system~(\ref{system}):}

Let us fix $(u(0),R(0))\in(0,+\infty)^{n+r}$ and consider the solution $(u(t),R(t))$ of~(\ref{system}).

\me We define
\begin{equation}
  \label{eq:def-G}
  G(u_1,\ldots,u_n,R_1,\ldots,R_r)=\sum_{i=1}^n (u_i-\bar u_i\log u_i)
  +\sum_{k=1}^r(R_k-\bar R_k\log R_k).
\end{equation}
Then, for any solution of~(\ref{system}) with $(n(0),R(0))\in(0,+\infty)^{n+r}$, one has
\[
\begin{split}
  \frac{d}{dt}G(u(t),R(t)) & =\sum_i (u_i-\bar u_i)\left(-d(x_i)+\sum_k\eta_k(x_i) R_k\right)\notag\\
&\quad  +\sum_k\frac{R_k-\bar R_k}{R_k}\left(g_k-R_k(1+\sum_i\eta_k(x_i)u_i)\right), \notag \\
\end{split}
\]
or 
\begin{equation}
\begin{split}
 \frac{d}{dt}G(u(t),R(t))  & =\sum_{i}(u_i-\bar u_i)\sum_k \eta_k(x_i)(R_k-\bar R_k)+\sum_i (u_i-\bar
  u_i)\left(-d(x_i)+\sum_k\eta_k(x_i)\bar R_k\right) \\
  & \qquad +\sum_k \frac{R_k-\bar R_k}{R_k}\left(g_k-R_k(1+\sum_i \eta_k(x_i)\bar u_i)\right)\\
&\qquad  -\sum_k(R_k-\bar R_k)\sum_i \eta_k(x_i)(u_i-\bar u_i). \end{split}\label{eq:calcul-calcul}
\end{equation}
Equation~(\ref{stable}) implies that the second term in the r.h.s.\ is non-positive. Therefore,
\begin{align}
  \frac{d}{dt}G(u(t),R(t)) & \leq
  \sum_k \frac{R_k-\bar R_k}{R_k}(\bar R_k-R_k)\left(1+\sum_i \eta_k(x_i)\bar u_i\right)\notag\\
&\qquad  +\sum_k \frac{R_k-\bar R_k}{R_k}\left(g_k-\bar R_k(1+\sum_i \eta_k(x_i)\bar u_i)\right)
  \notag \\
  & \leq-\sum_k\frac{(R_k-\bar R_k)^2}{R_k}\left(1+\sum_i \eta_k(x_i)\bar
    u_i\right), \label{eq:calcul2}
\end{align}
since the second term of the r.h.s. is zero. Therefore, $G$ is a Lyapunov functional for the system~(\ref{system}), degenerate in the
sense that, in view of~(\ref{eq:calcul-calcul}), its derivative could vanish when $R_k=\bar R_k$ but $u_i\not=\bar u_i$ for some
$i$.

\me Note that $G$ is convex (strictly in the $R_k$ and in the $u_i$ for which $\bar u_i\neq 0$) and $G(u,R)\rightarrow+\infty$
when $|u|+|R|\rightarrow+\infty$ in $\mathbb{R}_+^{n+r}$. As a consequence, the function $u(t)$ is bounded, say by $\bar U$.

\me Since $\dot{R}_k\leq g_k-R_k$, $R_k(t)$ is uniformly bounded in time, and by Lyapunov's Theorem, $R(t)\rightarrow \bar{R}$ when
$t\rightarrow+\infty$ and thus
$$
\alpha:=\inf_{t\geq 0}\inf_{k}R_k(t)>0.
$$
Now, let us define
$$
I:=\{i:\sum_k\eta_k(x_i)\bar R_k=d(x_i)\}.
$$
Note that if $\bar u_i>0$ then $i\in I$ but that $I$ could be larger. Since $(\bar u,\bar R)$ satisfies~(\ref{stable}),
$-d(x_i)+\sum_k\eta_k(x_i)\bar R_k<0$ for all $i\not\in I$, i.e.
%
\[
\beta:=\inf_{i\not\in I} \left(d(x_i)-\sum_k \eta_k(x_i)\,\bar R_k\right)>0.
\]
Then, using \eqref{eq:calcul-calcul}, one can find a more precise expression than \eqref{eq:calcul2}
\begin{equation}
  \label{eq:calcul3}
  \frac{d}{dt}G(u(t),R(t)) \leq-\beta\sum_{i\not\in I} u_i-\sum_k\frac{(R_k-\bar R_k)^2}{R_k}\left(1+\sum_i \eta_k(x_i)\bar
    u_i\right).
\end{equation}

\paragraph{Step 3. Second Lyapunov functional for the system~(\ref{system}):}

We can partly correct the problem that $G$ is a ``degenerate'' Lyapunov functional by slightly modifying it: let $\gamma$ be a small
positive number and define
$$
H(u,R)=G(u,R)+\gamma\tilde{G}(u,R),
$$
where
\begin{equation}
  \label{eq:def-tilde-G}
  \tilde{G}(u,R)=\sum_k(R_k-\bar R_k)\sum_i\eta_k(x_i)(u_i-\bar u_i).  
\end{equation}
Then
\begin{align}
 \frac{d}{dt}\tilde{G}(u(t),R(t))  & =\sum_k\left(g_k-R_k(1+\sum_j \eta_k(x_j)\bar u_j)\right)\sum_i \eta_k(x_i)(u_i-\bar u_i)
  \notag\\ & \qquad-\sum_k R_k\left(\sum_i\eta_k(x_i)(u_i-\bar u_i)\right)^2 \notag \\ 
  & \qquad+\sum_k (R_k-\bar R_k)\sum_i\eta_k(x_i)u_i\left(-d(x_i)+\sum_l \eta_l(x_i)\bar
    R_l\right) \notag 
  \\ & \qquad+\sum_i u_i\left(\sum_k\eta_k(x_i)(R_k-\bar R_k)\right)^2. \label{eq:calcul}
\end{align}
Since $(\bar u, \bar R)$ is an equilibrium of~(\ref{system}), the first term in the r.h.s.\ is equal to
$$
\sum_k \frac{g_k(\bar R_k-R_k)}{\bar R_k}\sum_i \eta_k(x_i)(u_i-\bar u_i) \leq \sum_k\left[\frac{g_k^2(\bar R_k-R_k)^2}{2\bar R_k^2
    R_k}+\frac{R_k}{2}\left(\sum_i\eta_k(x_i)(u_i-\bar u_i)\right)^2\right].
$$
It follows from the definition of the set $I$ that the third term of the r.h.s.\ of~(\ref{eq:calcul}) is equal to
$$
\sum_k (R_k-\bar R_k)\sum_{i\not\in I}\eta_k(x_i)u_i\left(-d(x_i)+\sum_l \eta_l(x_i)\bar R_l\right)\leq C\sum_{i\not\in I}u_i,
$$
where the constant $C$ depends only on uniform upper bound for the functions $R_k(t)$, $\eta_k(x)$ and $d(x)$. Finally, using
Cauchy-Schwartz inequality and the boundedness of $u$, the last term of the r.h.s.\ of~(\ref{eq:calcul}) is bounded by
$$
C'\sum_k\frac{(R_k-\bar R_k)^2}{R_k},
$$
where the constant $C'$ depends only on $\bar U$ and uniform upper bounds for $R_k(t)$ and $\eta_k(x)$.

\noindent Combining all these inequalities with~(\ref{eq:calcul3}), we obtain
\begin{align}
  \frac{d}{dt}H(u(t),R(t)) & \leq-\left(1-\gamma\frac{\sup_k g_k^2}{2\alpha^2}-\gamma
    C'\right)\sum_k\frac{(R_k-\bar R_k)^2}{R_k} \notag \\ &
  \qquad-\frac{\gamma}{2}\sum_k R_k\left(\sum_i\eta_k(x_i)(u_i-\bar u_i)\right)^2-(\beta-\gamma
  C)\sum_{i\not\in I}u_i \notag \\
  & \leq-\frac{1}{2}\sum_k\frac{(R_k-\bar R_k)^2}{R_k}-\frac{\gamma}{2}\sum_k R_k\left(\sum_i\eta_k(x_i)(u_i-\bar u_i)\right)^2-\frac{\beta}{2}
  \sum_{i\not\in I}u_i, \label{eq:last-eq}
\end{align}
if $\gamma$ is small enough. Lyapunov's Theorem then entails that the set of accumulation points of $(u(t),R(t))$ when
$t\rightarrow+\infty$ is contained in the sub-manifold $M$ of $\mathbb{R}_+^{n+r}$ defined as the set of
$(v_1,\ldots,v_n,S_1,\ldots,S_r)\in\mathbb{R}_+^{n+r}$ satisfying
\begin{gather*}
  S_k=\bar R_k,\qquad\forall 1\leq k\leq r \\
  \sum_{i=1}^n\eta_k(x_i)(v_i-\bar u_i)=0,\qquad\forall 1\leq k\leq r, \\
  v_i=0,\qquad\forall i\not\in I.
\end{gather*}
Since $\bar u_i=0$ for all $i\not\in I$, this sub-manifold $M$ contains the point $(\bar u,\bar R)$. Since $(\bar u,\bar R)$ is an
equilibrium, the second system of equations above is equivalent to
$$
\frac{g_k}{1+\sum_i\eta_k(x_i)v_i}=\frac{g_k}{1+\sum_i\eta_k(x_i)\bar u_i}=
\bar{R_k},\qquad\forall 1\leq k\leq r.
$$
Therefore Assumption~\eqref{A4} applied to the vector of traits $(x_i)_{i\in I}$ implies that $v_i=\bar u_i$ for all $i\in I$. Hence
$M$ is reduced to the point $(\bar u,\bar R)$ and the proof of Theorem~\ref{equilibrium} is completed.

\subsection{Some examples}
\label{sec:ex}

\subsubsection{The monomorphic case}
\label{sec:monomorphic}

In case $n=1$, we consider a unique trait $x$. The equilibrium $\bar u(x)$ defined in Theorem \ref{equilibrium} satisfies
\begin{equation}
  \label{eq:eq-mono}
  \sum_{k=1}^r \frac{\eta_{k}(x)\, g_{k}}{1+ \eta_{k}(x)\, \bar u(x)} = d(x).
\end{equation}
Remark that the left hand side is a decreasing function of $\bar u(x)$, and so there is a unique solution to this equation.
Therefore by the Implicit Function Theorem and Assumption \eqref{A1}, the function $\bar u(x)$ is a $C^2$-function.

\noindent The resources at equilibrium are thus given by
$$
\bar R_{k}(x)= \frac{ g_{k}}{1+ \eta_{k}(x)\, \bar u(x)}.
$$

\subsubsection{The dimorphic case}
\label{sec:dimorphic}

In case where $n=2$, we consider two distinct traits $x_{1}$ and $x_{2}$. Then the fitness function defined in \eqref{fitness} is
given by
\begin{equation}
  \label{fitness2}
  f(x_{2};x_{1}) = - d(x_{2}) + \sum_{k=1}^r \frac{\eta_{k}(x_{2})\,g_{k}}{1+ \eta_{k}(x_{1})\,\bar u(x_{1})}.
\end{equation}
It is immediate to observe that $f$ is $C^2$ with respect to $(x_1,x_2)$.

\noindent The system $\text{CH}(2,x_1,x_2)$ has two obvious equilibria: $E_{1}= (\bar u(x_{1}), 0, \bar R(x_{1}))$ and $E_{2}=(0,
\bar u(x_{2}), \bar R(x_{2}))$. In view of \eqref{stable}, $E_{1}$ is the equilibrium given by Theorem \ref{equilibrium} if and only
if $f(x_{2}; x_{1}) \leq 0$ and a similar condition for $E_{2}$ (and we cannot have both $f(x_2;x_1)\leq 0$ and $f(x_1;x_2)\leq 0$).
If we have $f(x_{2}; x_{1})> 0$ and $f(x_{1}; x_{2})> 0$, the equilibrium of Theorem \ref{equilibrium} must have positive density
coordinates $(\bar u_{1}(x_{1}, x_{2}), \bar u_{2}(x_{1}, x_{2}))$. Therefore, we have the following result.

\begin{prop}
  \label{prop:dimorph}
  For any $x_1\not=x_2$ in ${\cal X}$, the traits $x_1$ and $x_2$ coexist if and only if $f(x_1;x_2)>0$ and $f(x_2;x_1)>0$.
  
  \noindent Moreover the vector $(\bar u_{1}(x_{1}, x_{2}), \bar u_{2}(x_{1}, x_{2}))$ is solution of the system
  \begin{align}
    \label{equilibre2}
    & d(x_{1}) =  \sum_{k=1}^r
    \frac{\eta_{k}(x_{1})\,g_k}{1+\eta_k(x_1)\bar u_1+\eta_k(x_2)\bar u_2}, \\
    & d(x_{2}) =  \sum_{k=1}^r
    \frac{\eta_{k}(x_{2})\,g_k}{1+\eta_k(x_1)\bar u_1+\eta_k(x_2)\bar u_2}.
  \end{align}
\end{prop}

\subsubsection{The $n$-morphic case}
\label{sec:n-morphic}

When the population is composed of $n$ distinct traits $x_1,\ldots,x_n$, the case of coexistence can be characterized by induction
over $n$ by negating the stability condition~(\ref{stable}) of Theorem~\ref{equilibrium} for each trivial equilibrium of the system
$\text{CH}(n,x_1,\ldots,x_n)$. By trivial equilibrium, we mean any equilibrium with at least one null coordinate for population
densities.

\me In the case where $\mathbf{x}=(x_1,\ldots,x_n)$ coexist, the equilibrium densities satisfy
\begin{equation}
  \label{eq:n-morph-eq}
  d(x_i)=\sum_{k=1}^r\frac{g_k\eta_k(x_i)}{1+\sum_{j=1}^n\eta_k(x_j)\bar u_j(\mathbf{x})},
  \quad\forall i\in\{1,\ldots,n\},
\end{equation}
and the resources concentrations at equilibrium are given by
\begin{equation}
  \label{eq:n-morph-res-eq}
  \bar{R}_k(\mathbf{x})=\frac{g_k}{1+\sum_{j=1}^n\eta_k(x_j)\bar u_j(\mathbf{x})},
  \quad\forall k\in\{1,\ldots,r\}.
\end{equation}

\me We have the following regularity result for the function $\bar{\mathbf{u}}$ on ${\cal D}_n$.
\begin{lem}
  \label{lem:bar-u-continue}
  For all $n\geq 1$, the function $\bar{\mathbf{u}}$ is $C^2$ and bounded on ${\cal D}_n$. In addition, ${\cal D}_n$ is an open
  subset of ${\cal X}^n$.
\end{lem}

\begin{proof}
  Let us first prove that $\bar{\mathbf{u}}$ is $C^2$ on $\mathcal{D}_n$. For $\mathbf{x}=(x_1,\ldots,x_n)\in{\cal D}_n$, the
  equilibrium population densities $\bar u_i(\mathbf{x})$ are the (unique, by Assumption~\eqref{A4}) solutions of~(\ref{eq:n-morph-eq}).
  Defining the function $F=(F_1,\ldots,F_n)$ from ${\cal X}^n\times\RR^n$ to $\RR^n$ by
  $$
  F_i(x_1,\ldots,x_n,u_1,\ldots,u_n)=d(x_i)
  -\sum_{k=1}^r\frac{g_k\eta_k(x_i)}{1+\sum_{j=1}^n\eta_k(x_j)u_i},
  \quad\forall i\in\{1,\ldots,n\},
  $$
  we see that $\bar{\mathbf{u}}(\mathbf{x})$ is characterized by the equation $F(\mathbf{x},\bar{\mathbf{u}}(\mathbf{x}))=0$. The
  Jacobian matrix of $F$ with respect to $(u_1,\ldots,u_n)$ at $(\mathbf{x},\mathbf{u})=(x_1,\ldots,x_n,u_1,\ldots,u_n)$ is given by
  $$
  J_{\mathbf{u}} F(\mathbf{x},\mathbf{u})=\left(\sum_{k=1}^r
    \frac{g_k\eta_k(x_i)\eta_k(x_j)}{\left(1+\sum_j \eta_k(x_j)u_j\right)^2}\right)_{1\leq
    i,j\leq n}.
  $$
  For all vector $v\in\RR^n$,
  $$
  v^* J_{\mathbf{u}} F(\mathbf{x},\mathbf{u}) v=\sum_{k=1}^r\frac{g_k}{\left(1+\sum_j
      \eta_k(x_j)u_j\right)^2}\left[\sum_{i=1}^n v_i\eta_k(x_i)\right]^2.
  $$
  Hence the matrix $J_{\mathbf{u}} F(\mathbf{x},\mathbf{u})$ is invertible if and only if the orthogonal vector space of the family
  of vectors $\{\eta_1(\mathbf{x}),\ldots,\eta_r(\mathbf{x})\}$ in $\RR^n$ is $\{0\}$, where
  $\eta_k(\mathbf{x})=(\eta_k(x_1),\ldots,\eta_k(x_n))$. This is equivalent to the fact that the vector space spanned by
  $\eta_1(\mathbf{x}),\ldots,\eta_r(\mathbf{x})$ is $\RR^n$. This is implied by~(\ref{eq:cond-A4-ii}), which is itself a consequence
  of Assumption~\eqref{A4}  by Proposition~\ref{prop:A4}. Hence the fact that $\bar{\mathbf{u}}$ is $C^2$ on $\mathcal{D}_n$ follows from
  the Implicit Functions Theorem, as well as the fact that $\mathcal{D}_n$ is open.

  \noindent Let us now prove that $\bar{\mathbf{u}}$ is bounded on $\mathcal{D}_n$. This property for all $n\geq 1$ is equivalent to the fact
  that the function $\bar{\mathbf{u}}$ is bounded on $\mathcal{X}^n$ for all $n\geq 1$. We shall prove this last point. In the proof
  of Theorem~\ref{equilibrium}, the equilibrium $\mathbf{\bar u}(\mathbf{x})$ is characterized as the unique global minimizer on
  $\mathbb{R}_+^{n}$ of the function
  $$
  (u_1,\ldots,u_{n})\mapsto
  F(x_1,\ldots,x_{n},u_1,\ldots,u_{n}):=\sum_{i=1}^{n}d(x_i)u_i-\sum_{k=1}^r
  g_k\log\left(1+\sum_{j=1}^{n}\eta_k(x_j)u_j\right).
  $$
  Now $F(x_1,\ldots,x_{n},0,\ldots,0)=0$ for all $\mathbf{x}\in{\cal X}^{n}$ and
  $$
  F(x_1,\ldots,x_{n},u_1,\ldots,u_{n})\geq\underline{d}(u_1+\ldots+u_{n})-r\bar
  g\log(1+\bar\eta(u_1+\ldots+u_{n})),
  $$
  where $\underline{d}:=\min_{x\in{\cal X}}d(x)>0$, $\bar{g}:=\sup_{1\leq k\leq r}g_k$ and $\bar\eta:=\sup_k\|\eta_k\|_\infty$.
  Therefore, the function $\mathbf{\bar u}(\mathbf{x})$ is uniformly bounded on ${\cal X}^{n}$, and Lemma~\ref{lem:bar-u-continue} is
  proved.
\end{proof}

\noindent The coexistence criterion of Proposition~\ref{prop:dimorph} extends as follows to the $n$-morphic case.
\begin{prop}
  \label{prop:n-morph}
  For any $\mathbf{x}\in{\cal D}_n$ and $y\in{\cal X}$ such that for some $i\in\{1,\ldots,n\}$,\\
  $(x_1,\ldots,x_{i-1},y,x_{i+1},\ldots,x_n)\in{\cal D}_n$ (this holds for example if $y$ is
  close enough to $x_i$). Then $(\mathbf{x},y)=(x_1,\ldots,x_n,y)\in{\cal D}_{n+1}$ if and
  only if $f(y;\mathbf{x})>0$ and $ f(x_i;x_1,\ldots,x_{i-1},y,x_{i+1},\ldots,x_n)>0. $
\end{prop}

\begin{proof}
  In the first step of the proof of Theorem~\ref{equilibrium}, $\mathbf{\bar u}(\mathbf{x},y)$ was characterized as the unique global
  minimizer of the functional $F$ defined in~(\ref{Lyapounov}). By Theorem~\ref{equilibrium}, we also know that the unique global
  minimizer of $F$ on ${\cal X}^n\times\{0\}$ is $\mathbf{u}^*_1:=(\mathbf{\bar u}(\mathbf{x}),0)$ and the unique global minimizer of
  $F$ on ${\cal X}^{i-1}\times\{0\}\times{\cal X}^{n-i+1}$ is
  $$
  \mathbf{u}^*_2:=\Big(\bar u_1(\mathbf{x}^*),\ldots,\bar u_{i-1}(\mathbf{x}^*),0,\bar u_{i+1}(\mathbf{x}^*),\ldots,\bar
  u_n(\mathbf{x}^*),\bar u_i(\mathbf{x}^*)\Big),
  $$
  where $\mathbf{x}^*=(x_1,\ldots,x_{i-1},y,x_{i+1},\ldots,x_n)$. Now, since the derivative of $F$ with respect to the $n+1$-th
  coordinate at $\mathbf{u}^*_1$ is
  $$
  -f(y;\mathbf{x})<0
  $$
  and the derivative of $F$ with respect to the $i$-th coordinate at $\mathbf{u}^*_2$ is
  $$
  -f(x_i;x_1,\ldots,x_{i-1},y,x_{i+1},\ldots,x_n)<0,
  $$
  the unique global minimizer of $F$ must have positive coordinates, which is the definition of $(\mathbf{x},y)\in{\cal D}_{n+1}$.
\end{proof}




\section{Long time stability of the stochastic process}
\label{sec:GD}

Theorem~\ref{equilibrium} gives the asymptotic behavior of the deterministic chemostat
systems. When $K$ is large, one can see the individual-based model as a stochastic
perturbation of the deterministic chemostat system. One may then wonder to what extent this
perturbation modifies the long time stability of the system. The next result answers this
question and is going to be useful in Section~\ref{sec:PES}. 

\me
Let $x_1,\ldots,x_n$ be distinct points in ${\cal X}$ and $\varepsilon >0$ and denote by ${\cal M}$ the set of finite measures on ${\cal X}$. We define
$$
B_{\varepsilon}(\mathbf{x})= \Big\{(\nu,\mathbf{R})\in {\cal M}\times \mathbb{R}^r:
Supp(\nu)\subset \{x_1,\ldots,x_n\}, \ \forall i\  |\langle
\nu,\mathbf{1}_{\{x_i\}}\rangle-\bar u_i({\mathbf{x}})|<\varepsilon,
\|\mathbf{R}-\bar{\mathbf{R}}({\mathbf{x}})\|\leq \varepsilon\Big\}
$$

\begin{thm}
  \label{thm:GD} Let $x_1,\ldots,x_n$ be distinct points in ${\cal X}$. Assume that
  $\nu^K_0=\sum_{i=1}^n u^{K,i}_0\delta_{x_i}$ such that $(\mathbf{u}^{K}_0,\mathbf{R}^K(0))$
  a.s. belongs for $K$ large enough to a compact subset $S$ of
  $(0,+\infty)^n\times\mathbb{R}^r$.
  \begin{description}
  \item[\textmd{(i)}] For all $\varepsilon >0$, there exist $T_{\varepsilon},
    V_{\varepsilon}>0$ such that
    \begin{equation*}
      \lim_{{K\to \infty}}\mathbb{P}\Big(\forall t \in [ T_{\varepsilon}, e^{V_\varepsilon K}\wedge
      T^K_{{mut}}] ,\; (\nu^K_{t},\mathbf{R}^K(t))\in B_{\varepsilon}(\mathbf{x})\Big) = 1.
    \end{equation*}
  \item[\textmd{(ii)}] Define $I= \{i: \ \bar u_i({\mathbf{x}})=0\}$ and $T^K_{ext} =
    \inf\{t\geq 0: \ \forall i\in I,\ \langle \nu_t^{K},\mathbf{1}_{\{x_i\}}\rangle=0\} $.
    Assume that for all $i\in I$, $f(x_{i}, \mathbf{x})<0$. Then for all $\delta>0$, 
    \begin{equation*}
      \lim_{{K\to \infty}}\mathbb{P}(T^K_{ext}\wedge T^K_{{mut}}\leq (a+\delta)\log K) = 1,
    \end{equation*}
    where $a:= \sup_{i\in I} (1/|f(x_{i}, \mathbf{x})|)$.
  \end{description}
\end{thm}

\me Results as Theorem \ref{thm:GD}-(i) are often called ``problem of exit from an attracting
domain'' and can be solved using classical large deviation tools~\cite{FW84}. In the next
section, we will need a stronger (and non standard) large deviation result for a (small
enough) perturbed version of the piecewise deterministic process
$(\nu^K_t,\mathbf{R}^K(t),t\geq 0)$. Therefore we provide a specific proof of the next
result based on the Lyapunov functionals used in Theorem~\ref{equilibrium}.

\begin{prop}
  \label{prop:FW}
  Let $x_1,\ldots,x_n$ be distinct traits and assume that the support of $\nu^K_0$ is a subset
  of $\{x_1,\ldots,x_n\}$. Consider the population process $(\nu^K, \mathbf{R}^K)$ with
  perturbed resource dynamics:
  \begin{equation}
    \label{eq:pert-ress}
    \frac{d R^K_{k}(t)}{dt} = g_{k} - R^K_{k} - R^K_{k}\, \langle \nu^{K}_t, \eta_{k}\rangle +
    a^K_k(t),
  \end{equation}
  where $a^K_k$ is a $(\mathcal{F}_t)_{t\geq 0}$-predictable random process bounded by a
  constant $\eta$.

  \me Then for any $\varepsilon>0$ small enough, there exists $\eta_{0}>0$, $V_\varepsilon>0$
  and $\varepsilon''<\varepsilon$ such that, if $\eta<\eta_{0}$ and
  $(\nu_0^{K},\:\mathbf{R}^K(0))\in B_{\varepsilon''}(\mathbf{x})$, the time of exit  $T_{exit}$ of
  $(\nu_t^{K},\:\mathbf{R}^K(t))$ from $ B_{\varepsilon}(\mathbf{x})$ is bigger than
  $e^{V_\varepsilon K}\wedge T^K_{{mut}}$ with probability converging to 1.
\end{prop}

\paragraph{Proof of Theorem~\ref{thm:GD}}
By Theorem~\ref{equilibrium} and the continuity of the flow of solutions to the chemostat
system, for all $\varepsilon''>0$,
$$
T_{\varepsilon''}:=\sup_{(u,R)\in S}T_{\varepsilon''}(u,R)<+\infty,
$$
where $T_{\varepsilon''}(u,R)$ is the last entrance time in $B_{\varepsilon''/2}(\mathbf{x})$ of
the solution of $\text{CH}(n;\mathbf{x})$ with initial condition $(u,R)$. Then, applying
Theorem~\ref{largepop} on the time interval $[0,T_{\varepsilon''}]$, we obtain that
$\mathbb{P}(A_K)\rightarrow 1$ when $K\rightarrow+\infty$, where
$$
A_K:=\Big\{(\nu^K_{T_{\varepsilon''}},\mathbf{R}^K(T_{\varepsilon''}))\in
B_{\varepsilon''}(\mathbf{x})\Big\}.
$$
Then, choosing $\varepsilon''$ as in Proposition~\ref{prop:FW} and setting
$T_\varepsilon=T_{\varepsilon''}$, Point~(i) follows from the Markov property and
Proposition~\ref{prop:FW} (with $a^K(t)\equiv 0$).

\me For all $i\in I$, on the event $A_K$, during all the time interval
$[T_\varepsilon,e^{V_\varepsilon K}\wedge T^K_{{mut}}]$, the birth rate of each individual of
trait $x_i$ is bounded from above by $\sum_{k}\eta_k(x_i)\bar{R}_k+C\varepsilon$ for some
constant $C$. Therefore, the number $Z^i_t$ of individuals with trait $x_i$ at time
$T_\varepsilon+t$ is dominated on the event $A_K$ by a continuous time binary branching
process $\tilde{Z}^i_t$ with birth rate $\sum_{k}\eta_k(x_i)\bar{R}_k+C\varepsilon$,
death rate $d(x_i)$ and initial condition $\tilde{Z}^i_0=Z^i_{T_{\varepsilon}}=K\langle\nu^K_{T_{\varepsilon}},\mathbf{1}_{\{x_i\}}\rangle\leq \varepsilon'' K$.

\me It is well-known (cf. e.g.~\cite[p.\;109]{AN72}) that
$$
\PP(\tilde{Z}^i_t=0\mid
\tilde{Z}^i_0=1)=1-\frac{-f(x_i;\mathbf{x})-C\varepsilon}{d(x_i)\exp((-f(x_i;\mathbf{x})-C\varepsilon)t)-\sum_{k}\eta_k(x_i)\bar{R}_k+C\varepsilon}.
$$
Choosing $\varepsilon$ small enough for $-f(x_i;\mathbf{x})>2C\varepsilon$ and using the fact
that $\mathbb{P}(\tilde{Z}^i_t>0\mid\tilde{Z}^i_0=j)=
1-\mathbb{P}(\tilde{Z}^i_t=0\mid\tilde{Z}^i_0=1)^j$, it is immediate to check that, taking
$t_K= (1/|f(x_i;\mathbf{x})|+\delta)\log K$,
$$
\lim_{K\rightarrow+\infty}\sup_{0\leq j\leq
  K\varepsilon}\mathbb{P}(\tilde{Z}^i_{t_K}>0\mid\tilde{Z}^i_0=j)=0
$$
if $\varepsilon$ is small enough. Point~(ii) then easily follows.\hfill$\Box$
\bigskip

\paragraph{Proof of Proposition~\ref{prop:FW}}
In all the proof, we shall denote by $\bar u$ and $\bar R$ the vectors
$\bar{\mathbf{u}}(\mathbf{x})$ and $\bar{\mathbf{R}}(\mathbf{x})$.

\me Let us first observe that, because all the events involved in Proposition~\ref{prop:FW} are
${\cal F}_{T^K_{\text{mut}}}$ measurable, it is sufficient to prove Proposition~\ref{prop:FW}
for the population process where all birth rates with mutations are set to 0. Let us still
denote this process by $(\nu^K_t,\mathbf{R}^K(t))_{t\geq 0}$. Then, at each time $t$, the set
of living traits in the process is a subset of $\{x_1,\ldots,x_n\}$ and the model reduces to a
birth and death Markov chain in $(\frac{1}{K}\mathbb{Z}_+)^n$ for
$(N^K_i(t):=\langle\nu^K_t,\mathbbm{1}_{x_i}\rangle,1\leq i\leq n)_{t\geq 0}$, coupled
with~(\ref{eq:pert-ress}).

\me Recall the definition of the Lyapunov functional $H(u,R)$ for the chemostat system in the proof of Theorem~\ref{equilibrium}:
$$
H(u,R)=G(u,R)+\gamma \tilde{G}(u,R),
$$
where $G$ is defined in~(\ref{eq:def-G}) and $\tilde{G}$ in~(\ref{eq:def-tilde-G}), and
$\gamma>0$ can be arbitrary provided it is small enough.

\me Set $I:=\{i:\bar u_i=0\}$ and $J:=\{i:d(x_i)=\sum_k\eta_n(x_i)\bar R_k\}$. Note that
$J^c\subset I$ but in general, $J^c\not= I$.
Note that $G$ is linear in the variables $(u_i,i\in I)$ and clearly a strictly convex function  of the variables $(u_i,i \not\in I)$ and $R$
with diagonal Hessian matrix.
Hence, since $\tilde{G}(u,R)$ is a quadratic form, making $\gamma$ smaller if necessary,
\begin{multline}
  \label{eq:Gronwall-1}
  \|u-\bar u\|^2+\|R-\bar R\|^2\leq \sum_{i\not\in I}|u_i-\bar u_i|^2+\sum_{i\in
    I}u_i+\|R-\bar R\|^2 \\ \leq C_1\left[H(u,R)-H(\bar u,\bar R)\right]\leq C_1^2\Big(
    \sum_{i\not\in I}|u_i-\bar u_i|^2+\sum_{i\in I}u_i+\|R-\bar R\|^2\Big),
\end{multline}
for some constant $C_1>1$ and for all $(u,R)$ close enough to $(\bar u,\bar R)$.

\me By~(\ref{eq:last-eq}), if $(u(t),R(t))$ is any solution of
$\text{CH}(n;\mathbf{x})$,
\begin{align*}
  \frac{d}{dt}H(u(t),R(t)) & \leq -C_2\|R-\bar
  R\|^2-2C_2\sum_{k=1}^r\left(\sum_{i=1}^n\eta_k(x_i)(u_i-\bar u_i)\right)^2-2C_2\sum_{i\not\in
    J}u_i \\
  & \leq -C_2\|R-\bar
  R\|^2-C_2\sum_{k=1}^r\left(\sum_{i\in J} \eta_k(x_i)(u_i-\bar u_i)\right)^2-C_2\sum_{i\not\in
    J}u_i
\end{align*}
for a positive constant $C_2$ if $(u,R)$ close enough to $(\bar u,\bar R)$. Due to
Assumption~\eqref{A4}, as seen in the end of the proof of Theorem~\ref{equilibrium}, the
second term of the right hand side is zero iff $u=\bar u$. Therefore,
introducing
$$
C_3:=\inf\Big\{\sum_{k=1}^r\Big(\sum_{i\in J} \eta_k(x_i)(u_i-\bar u_i)\Big)^2:
  (u_i)_{i\in J}\in\mathbb{R}_+^{|J|} \text{\ s.t.\ }\sum_{i\in J}|u_i-\bar u_i|^2=1\Big\}>0,
$$
one deduces that
\begin{equation}
  \label{eq:Gronwall-2}
  \frac{d}{dt}H(u(t),R(t)) \leq -C_2\|R-\bar R\|^2-C_2C_3\sum_{i\in J}|u_i-\bar u_i|^2
-C_2\sum_{i\not\in
    J}u_i \leq -C_4\left[
    \|u-\bar u\|^2+\|R-\bar R\|^2\right]
\end{equation}
for some positive constant $C_4$.

\me In order to keep notations simple, we shall denote by $H(\nu^K_t,R^K(t))$ the function
$$
H\Big(N^K_1(t),\ldots,N^K_n(t),R^K_1(t),\ldots,R^K_r(t)\Big).
$$
Then, it follows from~(\ref{eq:constr-nu}) that
\begin{align*}
  & H(\nu^K_t,R^K(t)) =H(\nu^K_0,R^K(0)) \\ &
  +\int_0^t\int_{\mathbb{N}}\int_0^\infty
  \left(H\Big(\nu^K_{s-}+\frac{\delta_{x_i(\nu_{s-})}}{K},R^K(s)\Big)-H(\nu^K_{s-},R^K(s))\right) \\
  & \qquad\qquad\qquad\quad \mathbbm{1}_{\{i\leq n(\nu_{s-}),\
    \theta\leq(1-\mu_Kp(x_i(\nu_{s-})))\sum_{k=1}^r\eta_k(x_i(\nu_{s-}))R_k(s)\}}N_1(ds,di,d\theta)
  \\ & +\int_0^t\int_{\mathbb{N}}\int_0^\infty\int_{\mathbb{R}^{\ell}}
  \left(H\Big(\nu^K_{s-}+\frac{\delta_{x_i(\nu_{s-})+h}}{K},R^K(s)\Big)-H(\nu^K_{s-},R^K(s))\right)
  \\ & \qquad\qquad\qquad\quad \mathbbm{1}_{\{i\leq n(\nu_{s-}),\
    \theta\leq m(x_i(\nu_{s-}),h)\mu_Kp(x_i(\nu_{s-}))\sum_{k=1}^r\eta_k(x_i(\nu_{s-}))R_k(s)\}}N_2(ds,di,d\theta,dh)
  \\ & -\int_0^t\int_{\mathbb{N}}\int_0^\infty
  \left(H\Big(\nu^K_{s-}-\frac{\delta_{x_i(\nu_{s-})}}{K},R^K(s)\Big)-H(\nu^K_{s-},R^K(s))\right)
  \mathbbm{1}_{\{i\leq n(\nu_{s-}),\ \theta\leq d(x_i(\nu_{s-}))\}}N_3(ds,di,d\theta) \\ &
  +\sum_{k=1}^r\int_0^t\frac{\partial H}{\partial R_k}(\nu^K_s,R^K(s))\left(g_k-R^K_k(s)
    \left(1+\langle \nu^K_s,\eta_k\rangle\right)+a_k^K(t)\right)ds.
\end{align*}
Let $T_{\text{exit}}$ denote the first exit time of $(\nu^K(t),R^K(t))$ from
$B_\varepsilon(\mathbf{x})$. Assuming $t\leq T_{\text{exit}}$, one can make a second order
expansion of the quantities of the form
$$
\left(H\Big(\nu^K_{s-}\pm \frac{\delta_{y}}{K},R^K(s)\Big)-H(\nu^K_{s-},R^K(s))\right)
$$
appearing in the previous equation. Introducing the compensated Poisson point measures $\tilde{N}_i:=N_i-q_i$ for $i=1,2,3$, we
obtain, for all $t\leq T_{\text{exit}}$,
\begin{align}
  & H(\nu^K_t,R^K(t)) =H(\nu^K_0,R^K(0))+M^K_t \notag \\ &
  +\sum_{i=1}^n\int_0^t\frac{\partial H}{\partial
    u_i}(\nu^K_s,R^K(s))N^K_i(s)\left(-d(x_i)+\sum_{k=1}^r\eta_k(x_i)R^K_k(s)\right)ds
  +O\left(\frac{t}{K}\right) \notag \\ &
  +\sum_{k=1}^r\int_0^t\frac{\partial H}{\partial R_k}(\nu^K_s,R^K(s))\left(g_k-R^K_k(s)
    \left(1+\langle \nu^K_s,\eta_k\rangle\right)\right)ds +O(\eta t), \label{eq:Gronwall-3}
\end{align}
where $M^K_t$ is the local martingale
\begin{align*}
  M^K_t & :=\int_0^t\int_{\mathbb{N}}\int_0^\infty
  \left(H\Big(\nu^K_{s-}+\frac{\delta_{x_i(\nu_{s-})}}{K},R^K(s)\Big)-H(\nu^K_{s-},R^K(s))\right) \\
  & \qquad\qquad\qquad\quad \mathbbm{1}_{\{i\leq n(\nu_{s-}),\
    \theta\leq(1-\mu_Kp(x_i(\nu_{s-})))\sum_{k=1}^r\eta_k(x_i(\nu_{s-}))R_k(s)\}}\tilde{N}_1(ds,di,d\theta)
  \\ & +\int_0^t\int_{\mathbb{N}}\int_0^\infty\int_{\mathbb{R}^{\ell}}
  \left(H\Big(\nu^K_{s-}+\frac{\delta_{x_i(\nu_{s-})+h}}{K},R^K(s)\Big)-H(\nu^K_{s-},R^K(s))\right)
  \\ & \qquad\qquad\qquad\quad \mathbbm{1}_{\{i\leq n(\nu_{s-}),\ \theta\leq
    m(x_i(\nu_{s-}),h)\mu_Kp(x_i(\nu_{s-}))\sum_{k=1}^r\eta_k(x_i(\nu_{s-}))R_k(s)\}}\tilde{N}_2(ds,di,d\theta,dh)
  \\ & -\int_0^t\int_{\mathbb{N}}\int_0^\infty
  \left(H\Big(\nu^K_{s-}-\frac{\delta_{x_i(\nu_{s-})}}{K},R^K(s)\Big)-H(\nu^K_{s-},R^K(s))\right)
  \mathbbm{1}_{\{i\leq n(\nu_{s-}),\ \theta\leq
    d(x_i(\nu_{s-}))\}}\tilde{N}_3(ds,di,d\theta).
\end{align*}
Assuming $\varepsilon$ small enough, combining~(\ref{eq:Gronwall-1}),~(\ref{eq:Gronwall-2})
and~(\ref{eq:Gronwall-3}), we obtain for all $t\in[0,T\wedge T_\text{exit}]$ and for $K\geq
1/\eta$
\begin{multline}
  \|N^K(t)-\bar u\|^2+\|R^K(t)-\bar R\|^2\leq C_1\Big[C_1\Big(\sum_{i\not\in I}|N^K_i(0)-\bar
  u_i|^2+\sum_{i\in I}N^K_i(0)+\|R^K(0)-\bar R\|^2\Big)+\sup_{t\in[0,T]}|M^K_t|  \\
  -C_4\int_0^t\left(\|N^K(s)-\bar u\|^2+\|R^K(s)-\bar R\|^2-C_5\eta\right)ds\Big] \label{eq:last}
\end{multline}
for some constant $C_5>0$. Therefore, before $T\wedge T_\text{exit}$,
$\|N^K(s)-\bar u\|^2+\|R^K(s)-\bar R\|^2$ cannot stay larger than $2C_5 \eta$ on a time interval
larger than
\begin{equation}
  \label{eq:def-t-eta}
  T_\eta:=\frac{C_1(\sum_{i\not\in I}|N^K_i(0)-\bar
  u_i|^2+\sum_{i\in I}N^K_i(0)+\|R^K(0)-\bar R\|^2)+\sup_{t\in[0,T]}|M^K_t|}{C_4C_5\eta}.  
\end{equation}
This implies the following lemma.

\begin{lem}
  \label{lem:interm}
  \begin{description}
  \item[\textmd{(1)}] Fix $T,\eta>0$, define $T_\eta$ as in~(\ref{eq:def-t-eta}) and let
    $S_\eta$ denote the first hitting time of the $2C_5\eta$-neighborhood of $(\bar u,\bar R)$ by $(N^K(t),R^K(t))$. Then, on the event
    \begin{equation}
      \label{eq:event-interm}
      \left\{ T_\eta\leq T\wedge \frac{\varepsilon}{2C_1C_4C_5\eta}\right\},    
    \end{equation}
    we have
    $$
    \sup_{t\in[0,S_\eta]}(\|N^K(t)-\bar u\|^2+\|R^K(t)-\bar R\|^2)\leq C_1C_4C_5\eta
    T_\eta\leq\frac{\varepsilon}{2}
    $$
    and $S_\eta\leq T_\eta\wedge T_\text{exit}$ a.s.
  \item[\textmd{(2)}] In addition, on the same event, we also have
    $$
    \sup_{t\in[0,T\wedge T_\text{exit}]}(\|N^K(t)-\bar u\|^2+\|R^K(t)-\bar R\|^2)\leq
    C_1C_4C_5\eta (T_\eta+T)\leq\frac{\varepsilon}{2}+C_1C_4C_5\eta T.
    $$
    Therefore, under the additional condition that
    \begin{equation}
      \label{eq:cond-interm}
      \eta<\frac{\varepsilon}{2C_1C_4C_5T},
    \end{equation}
    we have $T_{\text{exit}}>T$.
  \end{description}
\end{lem}
We are also going to use exponential moment estimates on the martingale $M^K_t$.

\begin{lem}
  \label{lem:expo-moment}
  For all $\alpha>0$ and $T>0$, there exists a constant $V_{\alpha,T}>0$ such that for
  all $K$ large enough
  $$
  \mathbb{P}\Big(\sup_{t\in[0,T\wedge T_{\text{exit}}]}|M^K_t|>\alpha\Big)\leq 
  \exp(-KV_{\alpha,T}).
  $$
\end{lem}

\begin{proof}
  Since we are only dealing with events occurring before $T_\text{exit}$, we can assume without
  loss of generality that the martingale $M^K_t$ has all its jumps bounded by a constant times
  $1/K$, and that all its jump times are (some of the) jump times of Poisson point measures
  which occur at a rate bounded by a constant times $K$ (i.e.\ the Poisson point measures
  $N_1(ds,di,d\theta)\mathbbm{1}_{i\leq K u^*,\,\theta\leq r\bar\eta \bar g}\,$, $\,N_2(ds,di,d\theta,dh)\mathbbm{1}_{i\leq K
    u^*,\,\theta\leq r\bar\eta \bar g}\,$ and $\,N_3(ds,di,d\theta)\mathbbm{1}_{i\leq K u^*,\,\theta\leq \bar d}\,$, where $u^*:=1+\max_i \bar u_i$).
  Therefore, the previous result is a quite standard consequence of properties of exponential
  martingales for pure jump processes. For example, this is a consequence of the proof of
  Proposition~4.1 in~\cite{GM97}, where it is proved that there exists a constant $C>0$
  independent of $T$ and $\varepsilon$ such that for all $a>0$
  $$
  \mathbb{P}\Big(\sup_{t\in[0,T]}|M^K_t|>\varepsilon\Big)\leq 2\exp(-Ka\varepsilon+CKT\tau(Ca)),
  $$
  where $\tau(x)=e^x-1-x$. Since $\tau(x)\sim x^2/2$ when $x\rightarrow 0$,
  Lemma~\ref{lem:expo-moment} then follows by choosing $a>0$ small enough.
\end{proof}

Let us now introduce two parameters $\varepsilon'$ and $\varepsilon''$ such that
$\varepsilon''<\varepsilon'/2<\varepsilon'<\varepsilon$, to be determined later.
Let $\tau_0:=0$. For all $k\geq 1$ such that $\tau_{k-1}<T_\text{exit}$, we define
the stopping times
\begin{align*}
  \tau'_k & :=\inf\{t\geq \tau_{k-1}:(\nu^K_t,\mathbf{R}^K(t))\not\in B_{\varepsilon'/2}(\mathbf{x})\}, \\
  \tau_k & :=\inf\{t\geq \tau'_{k}:(\nu^K_t,\mathbf{R}^K(t))\in B_{\varepsilon''}(\mathbf{x}) \text{\ or\
  }(\nu^K_t,\mathbf{R}^K(t))\not\in B_{\varepsilon}(\mathbf{x})\}.
\end{align*}
Then, combining the strong Markov property with Lemmata~\ref{lem:interm}~(1)
and~\ref{lem:expo-moment}, it is not difficult to prove (cf.~\cite{DZ93}) that, for convenient
choices of $\varepsilon'$ and $\varepsilon''$, setting $\eta=\varepsilon''/2C_5$,
$T=2C_1\varepsilon^{\prime 2}/C_4C_5\eta$ and $\alpha=C_1\varepsilon^{\prime 2}$, these exists
a constant $V_\varepsilon:=V_{\alpha,T}$ such that
$$
\sup_{(\nu^K_0,\mathbf{R}^K(0))\in B_{\varepsilon'}(\mathbf{x})}\mathbb{P}(\tau_1<T_\text{exit})\geq
1-e^{-KV_\varepsilon}
$$
and thus, for all $k\geq 1$,
$$
\sup_{(\nu^K_0,\mathbf{R}^K(0))\in B_{\varepsilon'}(\mathbf{x})}\mathbb{P}(\tau_k<T_\text{exit} \mid
\tau_{k-1}<T_\text{exit})\geq 1-e^{-KV_\varepsilon}.
$$
Therefore, if $k_\text{exit}$ denotes the unique integer $k$ such that $\tau_k=T_\text{exit}$,
then $k_\text{exit}$ is larger or equal to a geometric random variable of parameter
$e^{-KV_\varepsilon}$. Hence $k_\text{exit}\geq e^{KV_\varepsilon/2}$ with a probability
converging to 1. Therefore, in order to complete the proof, it only remains to check that for
all $k$, $\tau'_k-\tau_{k-1}>T'$ with positive probability for some $T'>0$. Thanks to the
strong Markov property, this is implied by
$$
\inf_{(\nu^K_0,\mathbf{R}^K(0))\in B_{\varepsilon''}(\mathbf{x})}\mathbb{P}(\tau'_1>T')>0,
$$
which is a consequence of Lemmata~\ref{lem:interm}~(2) and~\ref{lem:expo-moment}, replacing
$\varepsilon$ by $\varepsilon'/2$ and choosing $T'=2C_1\varepsilon^{\prime\prime 2}/C_4C_5\eta$
and $\alpha=C_1\varepsilon^{\prime\prime 2}$ (in view of the choice of $\eta$ above, this can
impose to reduce $\varepsilon''$ so that $16C_1^2\varepsilon^{\prime\prime
  2}<\varepsilon^{\prime 2}$, but this only changes the constant $V_\varepsilon$ above), and
the proof of Proposition~\ref{prop:FW} is completed.\hfill$\Box$
\bigskip

\section{Rare mutations and evolutionary time scales}
\label{sec:PES}

We now go back to the stochastic model of Section~\ref{sec:model}. Our goal is to describe the effect of the random mutations on the
population under a scaling of large population and rare and small mutations. We  give three convergence results corresponding to three evolutionary time scales. These results are similar to the ones obtained  the Lotka-Volterra competition case ~(cf. \cite{CM11}) and the new ingredients needed in the proofs have been already obtained in the previous sections.
Therefore this section  is devoted to the statement of the results and their application to the example introduced in the Section~\ref{sec:example-1} and the rigorous modifications  of the proofs  are moved to Appendix.

\subsection{Convergence to the Polymorphic Evolution Sequence for chemostat system}
\label{sec:cv-PES}

As shown in Theorem~\ref{largepop}, mutations have no influence in the limit of large $K$ and small $\mu_K$ on the original
time-scale of the population process. We show that evolution proceeds on the longer time-scale of mutations $\frac{t}{K\mu_K}$, for which the population process converges to a
pure jump process describing successive mutant invasions. After each mutation, disadvantaged traits are driven to extinction due to
the competition induced by the deterministic chemostat systems. This convergence  to a pure jump process is illustrated by  Figure~\ref{fig:sim-mut-rares}.

\begin{figure}
\label{fig:sim-mut-rares} 
 \begin{minipage}[b]{0.47\textwidth}
   \begin{center}     
    \includegraphics[width=6cm]{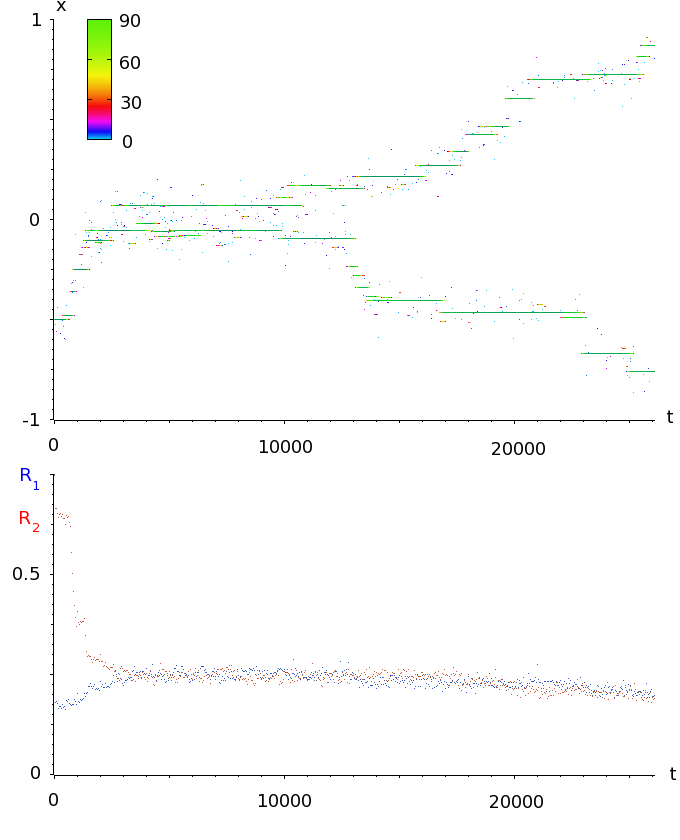}
    \medskip\\
    $K=300,\ p=0.0003,\ \sigma=0.06,\ a=1/4$
   \end{center}
  \end{minipage}
  \begin{minipage}[b]{0.47\textwidth}
   \begin{center}     
    \includegraphics[width=6cm]{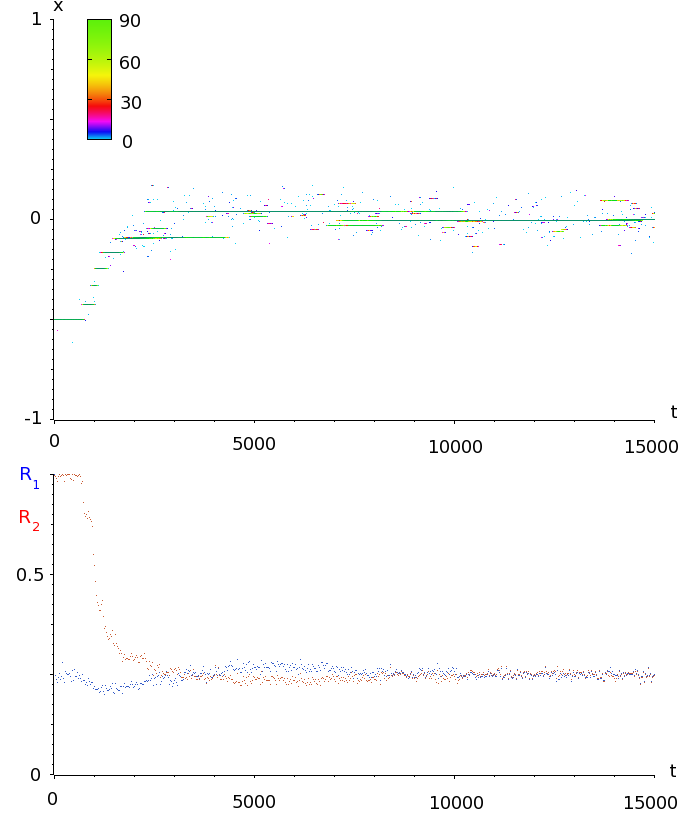}
    \medskip\\
    $K=300,\ p=0.0003,\ \sigma=0.06,\ a=1$
   \end{center}
  \end{minipage}
  \caption{Simulations of the individual-based model with rare mutations for two different values of the parameter $a$. Upper panels:
    time evolution of the trait density in the population. Lower panels: time evolution of the resources concentrations.}
\end{figure}

\noindent The limit process  takes values in the set of coupled equilibria of population measures and
resources, defined by
\begin{equation*}
  {\cal M}_0:=\left\{\left(\sum_{i=1}^n\bar{u}_i(\mathbf{x})\delta_{x_i},\ \mathbf{\bar R}(\mathbf{x})\right);
    \:n\geq 1,\:x_1,\ldots,x_n\in{\cal X}\mbox{\ coexist}\right\},
\end{equation*}
where $\mathbf{\bar u(x),\bar R(x)}$ were defined in Theorem~\ref{equilibrium}. Evolution proceeds by jumps of the support of the
population measure. \bigskip

\noindent We introduce the following assumption:
\begin{equation}
  \label{B}
  \begin{aligned}
    & \text{\it For all } (x_1,\ldots,x_n)\in\mathcal{D}_n, \text{\it\ for Lebesgue almost every trait\ } x_{n+1} \in {\cal
      X}, \text{\it\ with the}\\
    & \text{\it notation\ }\mathbf{x}=(x_1,\ldots,x_n, x_{n+1}), \text{\it\ we assume that for all\ } i\in \{1, \cdots, n+1\},
     \\ & \text{\it either\ } \bar u_i(\mathbf{x})>0 \text{\it\ or\ }
    -d(x_i)+\sum_{k=1}^r \eta_k(x_i)\bar R_k(\mathbf{x})< 0 \text{\it\ if\ } \bar u_i(\mathbf{x})=0.
  \end{aligned}
\end{equation}


\me In particular, this assumption means that, with the previous notations, defining
$I=\left\{i\in\{1,\ldots,n+1\}:-d(x_i)+\sum_{k=1}^r \eta_k(x_i)\bar R_k(\mathbf{x})< 0\right\}$, the traits $(x_i)_{i\not\in I}$
coexist.

\me This assumption implies that we will always be in a clear-cut, non degenerate situation: At equilibrium after the invasion of a
new mutant trait $x_{n+1}$ in a population with traits $x_1,\ldots,x_n$, for all of these traits, either it survives or the
corresponding growth rate (and eigenvalue of the Jacobian) is strictly negative. This allows to use Theorem \ref{thm:GD}-(ii).

\me
\begin{thm}
  \label{thm:PES-fdd}
  Assume 
  \eqref{A1},\eqref{A2},\eqref{A3},\eqref{A4},\eqref{B}. Take coexisting $\ x_1,\ldots,x_n\in{\cal X}$  and assume that $\nu^K_0=\sum_{i=1}^n
  u^K_i\delta_{x_i}$ with $u^K_i\rightarrow \bar{u}_i(\mathbf{x})$ in probability for all $1\leq i\leq n$. Assume also that
  $R_k^K(0)\rightarrow\bar{R}_k(\mathbf{x})$ in probability when $K\rightarrow+\infty$ for all $1\leq k\leq r$. Assume finally that
  \begin{equation}
    \label{eq:u_K-K}
    \forall V>0,\quad \log K\ll \frac{1}{K\mu_K}\ll \exp(VK), \quad \hbox{ as } K\to \infty.
  \end{equation}
  Then, $((\nu^K_{t/K u_K},\mathbf{R}^K(t/K\mu_k));t\geq 0)$ converges to the ${\cal M}_0$-valued Markov pure jump process
  $((\Lambda^\sigma_t,{\cal R}^\sigma(t)); t\geq 0)$ defined as follows: $\Lambda_0=\sum_{i=1}^n\bar{u}_i(\mathbf{x})\delta_{x_i}$, ${\cal
    R}^\sigma(t)=\mathbf{\bar R}(\text{Supp}(\Lambda^\sigma_t))$ where $\text{Supp}(\Lambda^\sigma_t)$ is the support
  of $\Lambda^\sigma_t$,
  and the process $\Lambda^\sigma$ jumps for all $j\in\{1,\ldots,n\}$
  $$
  \hbox{from }\ \sum_{i=1}^n \bar u_i(\mathbf{x})\delta_{x_i}\ \hbox{ to }\ \sum_{i=1}^{n} \bar u_i(x_1,\ldots,x_n,x_j+h)\delta_{x_i}
  +\bar u_{n+1}(x_1,\ldots,x_n,x_j+h)\delta_{x_j+h}
  $$
  with jump measure
  \begin{equation}
    \label{taux}
    p(x_j)\left(\sum_{k=1}^r\eta_k(x_j)\bar R_k(\mathbf{x})\right)\bar u_j(\mathbf{x})\frac{[f(x_j+h;\mathbf{x})]_+}
    {\sum_{k=1}^r\eta_k(x_j+h)\bar R_k(\mathbf{x})}m_\sigma(x_j,h)dh.
  \end{equation}
  The convergence holds in the sense of finite dimensional distributions on ${\cal M}_F$, the set of finite positive measures on
  $\mathcal{X}$ equipped with the topology of the total variation norm.
\end{thm}

\noindent Following~\cite{CM11}, we call this process \emph{Polymorphic Evolution Sequence}
(PES).   The
main steps of the proof are recalled in Appendix~\ref{sec:pf-PES}.
Note that the assumption \eqref{eq:u_K-K} is quite natural in view of Theorem \ref{thm:GD} and ensures that mutation occur after competition eliminates disadvantageous traits (Theorem \ref{thm:GD}-(ii)) and before the population densities drift away from equilibrium (Theorem \ref{thm:GD}-(i)).

\me As will appear in the proof, we may interpret the transition rates of the PES as follows: starting from an equilibrium population
$ \sum_{j=1}^n \bar u_j(\mathbf{x})\delta_{x_j}$, the process waits an exponential time of parameter
$$
\sum_{j=1}^n p(x_j) \left(\sum_{k=1}^r\eta_k(x_j)\bar R_k(\mathbf{x})\right)\bar u_j,
$$
interpreted as a mutation time. The trait of the parent is selected in the population as $x_I$, where $I$ is a random variable in
$\{1,\ldots,n\}$ with distribution
$$
\mathbb{P}(I=i)=\frac{p(x_i) \left(\sum_{k=1}^r\eta_k(x_i)\bar R_k(\mathbf{x})\right)\bar u_i} {\sum_{j=1}^n p(x_j)
  \left(\sum_{k=1}^r\eta_k(x_j)\bar R_k(\mathbf{x})\right)\bar u_j}.
$$
This distribution favors the traits with higher resource consumption.

\me Then, the mutant trait is given by $x_I+H$, where $H$ is distributed as $m_\sigma(x_I,h)dh$. An actual jump occurs if the mutant
population, initially composed of a single individual, does not go extinct and invades. A comparison argument between the mutant
population size and branching processes allows to compute the invasion probability as the survival probability of a branching
process, given by
$$
\frac{[f(x_I+H;\mathbf{x})]_+} {\sum_{k=1}^r\eta_k(x_I+H)\bar R_k(\mathbf{x})}.
$$
After the invasion, the new support of the process $\Lambda_t$ is given by the set of traits with nonzero densities in the vector of
densities $\mathbf{\bar u}(x_1,\ldots,x_n,x_I+H)$.


\subsection{The limit of small mutations: the canonical equation}
\label{sec:small-mut}

Until the end of the paper, we  assume by simplicity that the trait space is one-dimensional, i.e.\ ${\cal X}\subset\mathbb{R}$.

\label{sec:can-eq}

\noindent  Our goal is to study the PES under an additional  biological assumption of small mutations ($\sigma\to 0$), which is  standard in this
context~\cite{Mal96,DL96,diekmann-jabin-al-05}. 
We prove that when $\sigma$ tends to zero, the PES converges on the time scale $\frac{t}{\sigma^2}$,
to the solution of a (deterministic) ODE, called \emph{canonical equation of adaptive dynamics}, or simply \emph{canonical equation}.

\me 
By Theorem~\ref{thm:PES-fdd},   we get $\Lambda^\sigma_t=\bar
u(X^\sigma_t)\delta_{X^\sigma_t}$ as long as there is no coexistence of two traits in the population. Since  $m_\sigma(x,h)dh=\frac{1}{\sigma}m(x,\frac{h}{\sigma})dh$, the pure jump Markov process $(X^\sigma_t,t\geq 0)$ has the  infinitesimal generator\begin{equation}
  \label{eq:gene-TSS}
  A^\sigma\phi(x)=\int_{\cal X}(\phi(x+\sigma h)-\phi(x))[g(x+\sigma h;x)]_+ m(x,h)dh,
\end{equation}
where
$$
g(y;x)=p(x)\left(\sum_{k=1}^r\eta_k(x)\bar R_k(x)\right)\bar u(x)\frac{f(y;x)} {\sum_{k}\eta_k(y)\bar R_k(x)}.
$$
The jump process $(X^\sigma_t,t\geq 0)$ is often called ``trait substitution sequence'', or TSS~\cite{Mal96}.

\noindent By Proposition~\ref{prop:dimorph}, the first time of coexistence of two traits in the PES is given by
\begin{align*}
  \tau^\sigma := & \inf\{t\geq 0:|\text{Supp}(\Lambda^\sigma_t)|=2\} \\ = & \inf\{t\geq 0:f(X^\sigma_t;X^\sigma_{t-})>0\text{\ and\
  }f(X^\sigma_{t-};X^\sigma_t)>0\}.
\end{align*}
By Corollary~\ref{cor:coex}, coexistence may only occur in the neighborhood of points $x^*$ such that $\partial_1 f(x^*;x^*)=0$. This
leads to the following definition.

\begin{defi}
  We call a trait $x^*\in{\cal X}$ an evolutionary singularity if $\partial_1 f(x^*;x^*)=0$.
\end{defi}

\me  Since $g(x;x)=0$, the form of the
generator~(\ref{eq:gene-TSS}) suggests to scale time as $t/\sigma^2$ in order to obtain a non-trivial
limit generator when $\sigma\rightarrow 0$. 
\begin{thm}[Theorem~4.4 of~\cite{CM11}]
  \label{thm:can-eq}
  Assume~ \eqref{A1},\eqref{A2},\eqref{A3},\eqref{A4},\eqref{B} and that $\Lambda_0^\sigma=\bar{u}(x_0)\delta_{x_0}$ where $x_0$ is \emph{not} and
  evolutionary singularity. Let $x(t)$ be the solution of 
  \begin{equation}
    \label{eq:can-eq}
    \frac{dx(t)}{dt}=\int_{\mathbb{R}}h [h\partial_1g(x(t);x(t))]_+\ m(x(t),h)dh.
  \end{equation}
  such that $x(0)=x_0$. Then,
  \begin{description}
  \item[\textmd{(i)}]
    For any $T>0$,
    $$
    \lim_{\sigma\rightarrow 0}\mathbb{P}(\tau^\sigma>T/\sigma^2)=1,
    $$
    and there exists $\sigma_0$ such that for all $\sigma<\sigma_0$, the process $(X^\sigma_{t/\sigma^2},t\in[0,T])$ is a.s.\
    monotone.
  \item[\textmd{(ii)}] For any $T>0$, the process $(\Lambda^\sigma_{t/\sigma^2},t\in[0,T])$ converges as $\sigma\rightarrow 0$ to the
    deterministic process $(\bar{u}(x(t))\delta_{x(t)},t\in[0,T])$ for the Skorohod topology on $\DD([0,T],{\cal M}_0)$, where ${\cal
      M}_0$ is equipped with the weak topology.
   \end{description}
\end{thm}
Equation~(\ref{eq:can-eq}) is known as the \emph{canonical equation of adaptive dynamics}~\cite{DL96}.

\me This result follows from the convergence of the TSS process $X^\sigma_{t/\sigma^2}$ to $x(t)$ when $\sigma\rightarrow
0$~\cite[Thm.~4.1]{CM11}, from the fact that mutation jumps in the trait space are bounded by $\sigma\text{Diam}(\mathcal{X})$ (by
the definition of the jump measure $m_\sigma$) which converges to 0 when $\sigma \rightarrow 0$ and from a careful study of the solutions of Equation \eqref{eq:can-eq}.

\begin{rem}
  \label{rem:sym-CEAD}
  In the case when $m(x,\cdot)$ is a symmetrical measure on $\mathbb{R}$ for all $x\in{\cal X}$,
  Equation~\textup{(\ref{eq:can-eq})} gets the classical form, heuristically introduced in~\textup{\cite{DL96}},
  \begin{equation*}
    \frac{dx(t)}{dt}=\frac{1}{2}K(x(t))\partial_1g(x(t);x(t)),
  \end{equation*}
  where $K(x)$ is the variance of $m(x,h)dh$.
\end{rem}

\subsection{Small mutations and  evolutionary branching}
\label{sec:evol-branch}

We will now introduce the last tools to predict the different behaviors observed in Figure~\ref{sim1} and in particular to characterize the diversification phenomenon of evolutionary branching. Since it is not captured by the canonical equation, evolutionary branching can only occur on a longer time scale and in the neighborhood of equilibria of the canonical equation. 
 Let $x^*:=\lim_{t\rightarrow+\infty}x(t)$ ( well-defined since ${\cal X}$ is compact). We make the  additional assumption
\emph{\be
\label{C} x^* \hbox{ is  in the interior of } {\cal X}, \int_{\mathbb{R}_-}m(x^*,h)dh>0 \hbox{  and } \int_{\mathbb{R}_+}m(x^*,h)dh>0.
\ee}
This assumption means that mutations are always possible from  $x^*$.
Then  in view of the canonical equation \eqref{eq:can-eq},
 $\partial_1 f(x^*; x^*)=0$, i.e.\ $x^*$ is an evolutionary singularity.

\me
The linear stability condition for the equilibrium $x^*$ implies that
$$
\partial_{11}f(x^*;x^*)+\partial_{12}f(x^*;x^*)\leq 0.
$$
Differentiating twice the equation $f(x;x)=0$ implies that
$$
\partial_{11} f(x;x)+2\partial_{12}
  f(x;x)+\partial_{22}f(x;x)=0,\quad\forall x\in{\cal X},$$
and so
$$
\partial_{11}f(x^*;x^*)\leq\partial_{22}f(x^*;x^*).
$$
We shall leave the degenerate case
$\partial_{22}f(x^*;x^*)=\partial_{11}f(x^*;x^*)$ for further studies, and assume below that
\begin{equation}
  \label{eq:hyp-evol-br}
  \partial_{11}f(x^*;x^*)<\partial_{22}f(x^*;x^*).  
\end{equation}
\bigskip

Let us recall the definition of evolutionary branching introduced in \cite{CM11}.

\begin{defi}
  \label{def:br}
  Fix $\sigma>0$ and $x^*$ an evolutionary singularity. For all $\eta>0$, we say that there is $\eta$-branching at $x^*$ for the PES
  $\Lambda^\sigma$ if
  \begin{itemize}
  \item there exists $t_1>0$ such that the support of the PES at time $t_1$ is composed of a single point belonging to
    $[x^*-\eta,x^*+\eta]$;
  \item there exists $t_2>t_1$ such that the support of the PES at time $t_2$ is composed of exactly $2$ points separated by a
    distance of more than $\eta/2$;
  \item between $t_1$ and $t_2$, the support of the PES is always a subset of $[x^*-\eta,x^*+\eta]$, and is always composed of at
    most $2$ traits, and has nondecreasing (in time) diameter.
  \end{itemize}
\end{defi}
This definition only considers \emph{binary} evolutionary branching. The next theorem proves that, in the neighborhood of an
evolutionary singularity, evolutionary branching of a mo\-no\-mor\-phic population into a $n$-morphic population with $n\geq 3$ is
impossible (at least when the trait space is one-dimensional). Note that the notion of evolutionary branching requires the
coexistence of two traits, but also that these two traits diverge from one another.

\begin{thm}[Evolutionary branching criterion]
  \label{thm:br}
  Assume~\eqref{A1},\eqref{A2},\eqref{A3},\eqref{A4},\eqref{B} and~\eqref{C}. Assume also that $\Lambda^\sigma_0=\bar{u}(x_0)\delta_{x_0}$ and that the
  canonical equation~(\ref{eq:can-eq}) with initial condition $x_0$ converges to an evolutionary singularity $x^*$ in the interior of
  ${\cal X}$. Assume finally that $x^*$ satisfies~(\ref{eq:hyp-evol-br}) and
  \begin{equation} 
    \label{eq:cond2}
    \partial_{22}f(x^*;x^*)+\partial_{11}f(x^*;x^*)\not=0.
  \end{equation}
  Then, for all sufficiently small $\eta$, there exists $\sigma_0>0$ such that for all $\sigma<\sigma_0$,
  \begin{description}
  \item[\textmd{(a)}] if $\partial_{11}f(x^*;x^*)>0$, $\PP(\eta\mbox{-branching at\ } x^* \mbox{\ for\ }
    \Lambda^\sigma)=1$.
  \item[\textmd{(b)}] if $\partial_{11}f(x^*;x^*)<0$, $\PP(\eta\mbox{-branching at\ } x^* \mbox{\ for\ } \Lambda^\sigma)=0$.
    Moreover,
    $$
    \PP\big(\forall t\geq\theta_\eta^\sigma(x^*),\ \mbox{\textup{Card}}
    (\mbox{\textup{Supp}}(\Lambda^\sigma_{t}))\leq 2\ \mbox{and}\ \mbox{\textup{Supp}}(\Lambda^\sigma_{t})
    \subset(x^*-\eta,x^*+\eta)\:\big)=1,
    $$
    where
    $$
    \theta^\sigma_{\eta}(x^*)=\inf\{t\geq 0,\ \mbox{\textup{Supp}}(\Lambda^\sigma_t)\cap(x^*-\eta,x^*+\eta)
    \not=\emptyset\}.
    $$
  \end{description}
\end{thm}
This criterion appeared for the first time in~\cite[Section 3.2.5]{Mal96} with an heuristic justification. Locally
around $x^*$, one of the two following events can occur almost surely: either there is binary evolutionary branching  or not. Coexistence can occur in 
both cases, but our proof shows that, after coexistence, the diameter of the support of the PES is a.s. non-decreasing (resp. non-increasing) in the first (resp. second) case.

\begin{rem}
  \label{prop:coex}
  It can be proved using Proposition~\ref{prop:sign-fitn-2}~(i)  that 
  \begin{itemize}
  \item if $\partial_{11}f(x^*;x^*)+\partial_{22}f(x^*;x^*)>0$, then for all neighborhood
    ${\cal U}$ of $x^*$ in $\mathcal{X}$, there exist $x,y\in{\cal U}$ that coexist and
  \item if $\partial_{11}f(x^*;x^*)+\partial_{22}f(x^*;x^*)<0$, then there exists a
    neighborhood ${\cal U}$ of $x^*$ in $\mathcal{X}$ such that any $x,y\in{\cal U}$ do not
    coexist.
  \end{itemize}
  Comparing this result with the criterion of Theorem~\ref{thm:br}, one has of course that
  evolutionary branching implies coexistence, but the converse is not true.
\end{rem}

\noindent As in~\cite{CM11}, the proof of Theorem~\ref{thm:br} is based on local expansions of
the fitness functions $f(y;x)$ and $f(z;x,y)$ in the neighborhood of $x^*$, given in our case
in Propositions~\ref{prop:sign-fitn-1} and~\ref{prop:sign-fitn-2} in the Appendix. The main
steps of the proof of Theorem \ref{thm:br} are given in Appendix \ref{app:br}.

\subsection{Back to our example}
\label{sec:example-bis}

Let us come back to the example developed in Subsection~\ref{sec:example-1}. The death rate of an individual with trait $x$ has the
form $d(x) = \frac{1}{2}+ ax^2$ with $a>0$. Let us compute in this case the quantities we are interested in. By symmetry
considerations, we have that $x^*=0$ is an evolutionary singularity. It then follows from~(\ref{eq:eq-mono}) that $\bar u(0)=3$ and
from~(\ref{fitness}) that
$$
\partial_{11}f(0;0)=-d''(0)+\frac{\eta''_1(0)+\eta''_2(0)}{1+\eta_1(0)\bar u(0)}=1-2a.
$$
For symmetry reasons again, we also have $\bar u'(0)=0$, from which we deduce that
$$
\partial_{22}f(0;0)=\frac{3}{2}-\frac{\bar u''(0)}{8}.
$$
Differentiating twice the relation
$$
\left(\frac{1}{2}+ax^2\right)(1+(x-1)^2\bar u(x))(1+(x+1)^2\bar u(x))=(x-1)^2(1+(x+1)^2\bar u(x))+(x+1)^2(1+(x-1)^2\bar u(x)),
$$
we obtain $\bar u''(0)=-4-16a$, which give $\partial_{22}f(0;0)=4+2a$.

\me Therefore, the inequality $\partial_{22}f(0;0)>\partial_{11}f(0;0)$ is always satisfied, and there is evolutionary branching
(resp.\ no evolutionary branching) at $x^*=0$ if $a<1/2$ (resp.\ $a>1/2$). In the case where $a<1/2$, the death rate does not
increase too much as the trait $x$ deviates from $0$. So the population is better off dividing into two specialized populations, each
consuming better either resource $1$ or resource $2$. In addition, since $\partial_{11}f(0;0)+\partial_{22}f(0;0)=5>0$, coexistence
is always possible in the neighborhood of $x^*=0$. Therefore, the simulations of Figures~\ref{sim1} and~\ref{fig:sim-mut-rares} are
consistent with Theorem~\ref{thm:br} and Proposition~\ref{prop:coex}.

\bi
{\bf Acknowledgements:} This work benefited from the support of the ANR MANEGE (ANR-09-BLAN-0215) and from the Chair "Mod\'elisation math\'ematique et biodiversit\'e" of Veolia Environnement - Ecole polytechnique - Mus\'eum national d'Histoire naturelle - Fondation X.

\appendix
\section*{Appendix}

\section{Sketch of the proof of Theorem~\ref{thm:PES-fdd}}
\label{sec:pf-PES}

We now give the general idea of the proof of Theorem~\ref{thm:PES-fdd}, extending the biological heuristics of~\cite{Mal96},
rigorously proved in~\cite{C06} in the case of Lotka-Volterra competition.

\me Let us roughly describe the successive steps of mutation, invasion and competition. The two steps of the invasion of a mutant in
a given population are firstly the stabilization of the resident population before the mutation and secondly the invasion of the
mutant population after the mutation.

\me Fix $\eta>0$. In the first step, assuming that $n$ traits $x_1,\ldots,x_n$ coexist,  Assumption $\frac{1}{K \mu_K}\ll e^{VK}$ and Theorem \ref{thm:GD} ensure that the population
densities $(\langle\nu^K_t,\mathbf{1}_{\{x_1\}}\rangle,\ldots,\langle\nu^K_t,\mathbf{1}_{\{x_n\}}\rangle)$ and the vector of
resources $R^K(t)$ belong to the $\eta$-neighborhood of $(\mathbf{\bar u}(\mathbf{x}),\mathbf{\bar R}(\mathbf{x}))$ with high
probability for large $K$ until the next mutant $y$ appears.  Therefore, until $T^K_{mut}$, the  population densities are roughly constant and the rate of mutation is approximated by $\mu_Kp(x_i)\sum_k\eta_k(x_i)\bar R_k(\mathbf{x})K\bar
u_i(\mathbf{x})$. More precisely, the next lemma can be proved
with  the same arguments as in  Lemma~2~(b) and~(c) in~\cite{C06}.
\begin{lem}
  \label{lem:bef-mut}
  Let $\ \mbox{Supp}(\nu^K_0)=\{x_1,\ldots,x_n\}\ $ that coexist. There exists
  $\varepsilon_0>0$ such that, if $( \nu_0^{K},\mathbf{R}^K(0))\in B_{\varepsilon_{0}}(\mathbf{x})$, then, for any
  $\varepsilon<\varepsilon_0$,
  \begin{gather*}
    \lim_{K\rightarrow+\infty}\PP\Big(T^K_{\text{mut}}>\log K;\ \forall t\in[\log K,T^K_{\text{mut}}] \,,\, (
    \nu_t^{K}, \mathbf{R}^K(t)) \in B_{\varepsilon}(\mathbf{x})\Big)=1,
    \\ K\mu_K T^K_{\text{mut}}\ \overset{{\cal L}}{\underset{K\rightarrow \infty}{\Longrightarrow}}\
    \mbox{\textup{Exp}}\Big(\sum_{j=1}^n p(x_j)\sum_{k=1}^r\eta_k(x_j)\bar R_k(\mathbf{x})\Big)
  \end{gather*}
  and
  $$
  \lim_{K\rightarrow+\infty}\PP(\mbox{at time $T^K_\text{mut}$, the mutant is born from trait
    $x_i$})=\frac{p(x_i) \sum_{k=1}^r\eta_k(x_i)\bar R_k(\mathbf{x})}{\sum_{j=1}^n p(x_j)\sum_{k=1}^r\eta_k(x_j)\bar
    R_k(\mathbf{x})}
  $$
  for all $i\in\{1,\ldots,n\}$, where $\:\overset{{\cal L}}{\Longrightarrow}\:$ stands for the convergence in distribution of real
  random variables and $\mbox{\textup{Exp}}(a)$ denotes the exponential distribution with parameter $a$.
\end{lem}

\noindent In the second step, we divide the invasion of a given mutant trait $y$ into 2 phases
shown in Fig.~\ref{fig:inv-fix}.

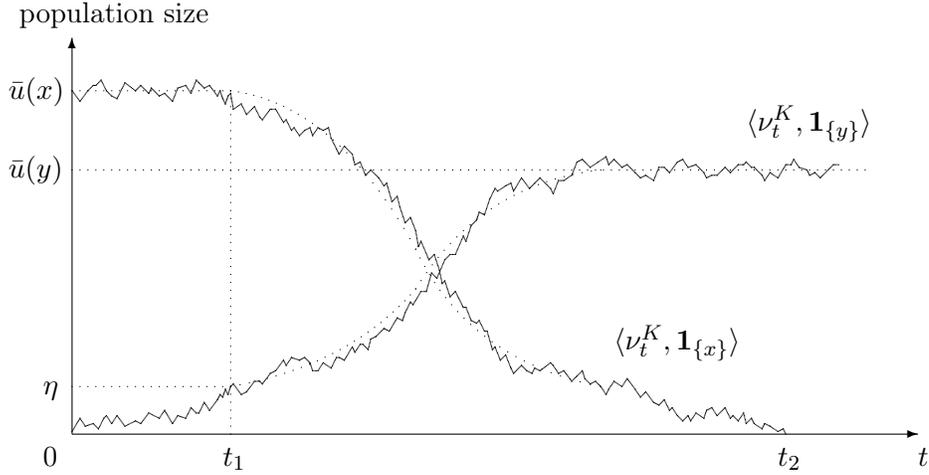
\begin{figure}[ht]
  \begin{center}
    \begin{picture}(350,180)(-20,-10)
      \put(0,0){\vector(1,0){320}} \put(0,0){\vector(0,1){150}}
      \put(-11,-12){0} \put(-11,15){$\eta$}
      \dottedline{3}(0,18)(60,18)
      \put(-24,97){$\bar{u}(y)$} \put(-24,127){$\bar{u}(x)$}
      \dottedline{3}(0,100)(300,100) \dottedline{3}(0,130)(60,130)
      \put(-20,156){population size} \put(57,-12){$t_1$}
      \put(267,-12){$t_2$} \put(320,-12){$t$}
      \dottedline{3}(60,0)(60,130) 
      \qbezier[27](60,130)(100,125)(130,70)
      \qbezier[25](130,70)(150,28)(200,18)
      \qbezier[25](60,18)(105,26)(130,55)
      \qbezier[25](130,55)(155,95)(200,100)
      \dottedline{0.5}(0,1)(3,6)(5,3)(8,4)(10,2)(12,6)(13,7)(15,4)(17,7)%
(20,3)(24,5)(26,4)(29,8)(33,6)(35,9)(38,4)(40,7)(43,6)(45,10)(47,8)(48,8)%
(50,12)(52,9)(55,13)(56,15)(57,14)(58,17)(59,15)(60,18)(61,19)(64,15)%
(66,16)(69,20)(71,20)(72,19)(75,23)(77,24)(80,28)(81,26)(83,29)(86,28)%
(88,29)(90,27)(91,24)(93,24)(95,27)(98,25)(100,29)(101,29)(103,28)(106,33)%
(107,31)(109,34)(110,34)(112,32)(113,36)(115,37)(118,40)(120,39)(123,44)%
(125,43)(126,46)(128,50)(129,49)(132,54)(134,54)(136,61)(138,59)(141,66)%
(143,68)(145,68)(148,75)(149,73)(151,79)(153,81)(154,84)(156,83)(158,89)%
(160,93)(161,92)(164,94)(166,95)(167,92)(170,97)(173,93)(175,95)(177,97)%
(181,93)(182,91)(184,96)(186,96)(189,98)(191,103)(193,104)(196,101)(199,103)%
(202,105)(203,102)(205,104)(208,101)(210,102)(213,99)(215,97)(217,98)%
(220,96)(222,101)(224,101)(226,99)(229,103)(231,104)(233,101)(236,102)%
(239,99)(240,97)(242,98)(244,98)(247,101)(249,99)(252,102)(254,102)(255,103)%
(257,101)(258,101)(261,98)(262,96)(264,99)(267,97)(270,102)(272,102)%
(273,104)(275,101)(278,100)(279,98)(281,99)(283,97)(288,102)(290,102)
      \put(255,115){$\langle\nu^K_t,\mathbf{1}_{\{y\}}\rangle$}
      \dottedline{0.5}(0,130)(3,126)(6,130)(8,132)(9,132)(11,134)(13,130)%
(15,128)(17,127)(20,133)(24,130)(26,132)(29,129)(30,131)(33,128)(35,129)%
(38,126)(40,131)(43,132)(45,129)(47,134)(48,133)(50,131)(52,132)(55,128)%
(56,130)(57,127)(58,125)(59,128)(60,129)(61,123)(64,125)(66,121)(69,124)%
(71,120)(72,119)(75,123)(77,123)(80,118)(81,119)(83,115)(84,116)(86,113)%
(88,115)(90,116)(91,113)(93,116)(95,115)(97,116)(98,112)(100,108)(101,106)%
(103,109)(106,106)(107,102)(109,103)(111,98)(113,100)(116,97)(118,94)%
(119,95)(121,88)(123,90)(124,86)(126,80)(128,82)(130,77)(131,71)(132,73)%
(135,66)(137,68)(139,62)(140,57)(141,58)(143,52)(145,54)(148,48)(149,48)%
(151,43)(153,42)(154,43)(156,39)(157,39)(159,32)(160,33)(162,28)(164,30)%
(166,25)(167,23)(170,26)(172,24)(174,23)(175,25)(177,24)(181,27)(182,26)%
(184,23)(186,24)(189,20)(191,22)(193,24)(194,21)(196,19)(198,20)(200,18)%
(202,15)(203,17)(205,20)(208,18)(210,21)(213,17)(215,14)(217,16)(220,12)%
(222,11)(224,13)(226,9)(227,10)(229,7)(231,5)(233,7)(234,5)(236,8)%
(239,6)(240,9)(242,10)(244,7)(247,10)(249,6)(250,7)(252,5)(254,8)%
(255,4)(257,6)(258,4)(261,2)(262,4)(264,3)(267,1)(269,2)(270,0)
      \put(205,32){$\langle\nu^K_t,\mathbf{1}_{\{x\}}\rangle$}
    \end{picture}
  \end{center}
  \caption{{\small The two steps of the invasion of a mutant trait $y$ in a monomorphic population with trait $x$.}}
  \label{fig:inv-fix}
\end{figure}

\noindent The first phase stops either when the mutant population gets extinct or reaches a
fixed small density $\eta>0$ (at time $t_1$ in Fig.~\ref{fig:inv-fix}). During all this phase,
the mutant density is small, and so, using the perturbed version of the large deviation result
(Proposition~\ref{prop:FW}), we can prove that the resident population stays close to its
equilibrium density $\bar{\mathbf{u}}(\mathbf{x})$ and the resource concentrations stay close
to $\bar{\mathbf{R}}(\mathbf{x})$. Therefore, similarly as in the proof of
Theorem~\ref{thm:GD}-(ii), the number of mutant individuals can be compared with branching
processes with birth rate $\sum_k\eta_k(y)\bar R_k(\mathbf{x})\pm C\varepsilon$ and death rate
$d(y)$. The growth rate of this branching process is close to the fitness $f(y; \mathbf{x})$,
which hence describes the ability of the initially rare mutant trait $y$ to invade the
equilibrium resident population with traits $x_1,\ldots,x_n$. If this fitness is positive
(i.e.\ if the branching processes are super-critical), the probability that the mutant
population reaches density $\eta>0$ at some time $t_1$ is close to the probability that the
branching process reaches $\eta K$, which is itself close to its survival probability, i.e.
close to $[f(y;\mathbf{x})]_+/\left(\sum_k\eta_k(y)\bar R_k(\mathbf{x})\right))$ when $K$ is
large. In addition, the comparison with branching processes shows that this first phase ends
before a time of order $\log K$ similarly as in the proof of Theorem~\ref{thm:GD}-(ii) (see
the proof of Lemma~3 in~\cite{C06}), and so the assumption
\begin{equation}
  \label{eq:hyp-log-K}
  \log K \ll \frac{1}{Ku_K}  
\end{equation}
ensures that no new mutation occurs during this first phase with high probability.

\me If the mutant population invades (i.e. reaches the density $\eta$), the second phase
stops when the population densities and the resource concentrations reach
$B_{\eta}(\bar{\mathbf{x}})$, where $\bar{\mathbf{x}}=(x_1,\ldots,x_n,y)$, and when the traits
$i$ such that $\bar u_i(\bar{\mathbf{x}})=0$ go extinct in the population (at time $t_2$ in
Fig.~\ref{fig:inv-fix}). Theorem~\ref{thm:GD}-(i) and-(ii) ensure that this phase is completed
with a probability converging to one after a time of order $\log K$. Note that
Assumption~(\ref{B}) is required to be able to apply Theorem~\ref{thm:GD}-(i) for almost all
mutant trait $y$ born in a population with traits $x_1,\ldots,x_n$. Again, the
assumption~(\ref{eq:hyp-log-K}) ensures that no new mutation occurs during this second phase
with high probability.




\me
These two phases are summarized in the following lemma, which can be proved as Lemma~A.4
in~\cite{CM11}.

\begin{lem}
  \label{lem:aft-mut}
  Let $\mbox{Supp}(\nu^K_0)=\{x_1,\ldots,x_n,y\}$ where $x_1,\ldots,x_n$ coexist and $y$ is a mutant trait that satisfy
  Assumption~\eqref{B}. We shall denote $x_{n+1}=y$ and $\bar{\mathbf{x}}=(x_1,\ldots,x_n,x_{n+1})$ for convenience. We define
  \begin{align*}
    \tau^K_1&=\inf\{t\geq 0: (\nu_t^{K},\mathbf{R}^K(t))\in B_{\varepsilon}(\bar{\mathbf{x}}) \mbox{\ and\
    }\forall i\text{\ s.t.\ }\bar u_i(\bar{\mathbf{x}})=0,\ \langle \nu_t^{K},\mathbf{1}_{\{x_i\}}\rangle=0\} \\ 
    \tau^K_2 &=\inf\{t\geq 0:\langle \nu_t^{K},\mathbf{1}_{\{y\}}\rangle=0 \text{ and  }
    ( \nu_t^{K},\mathbf{R}^K(t))\in B_{\varepsilon}(\mathbf{x}) \}.
  \end{align*}
  Assume that $\langle \nu_0^{K},\mathbf{1}_{\{y\}}\rangle=1/K$ (a single initial mutant). Then, there exists $\varepsilon_0$ such
  that for all $\varepsilon<\varepsilon_0$ and  if $( \nu_0^{K},\mathbf{R}^K(0))\in B_{\varepsilon}(\mathbf{x})$,
  \begin{gather*}
    \lim_{K\rightarrow+\infty}\PP(\tau^K_1<\tau^K_2)
    =\frac{[f(y;\mathbf{x})]_+}{\sum_k \eta_k(y)\bar R_k(\mathbf{x})},\quad
    \lim_{K\rightarrow+\infty}\PP(\tau^K_2<\tau^K_1)
    =1-\frac{[f(y;\mathbf{x})]_+}{\sum_k \eta_k(y)\bar R_k(\mathbf{x})} \\
     \mbox{and}\quad\forall \eta>0,\quad
    \lim_{K\rightarrow+\infty}\PP\left(\tau^K_1\wedge\tau^K_2<\frac{\eta}{K
      u_K}\wedge T^K_\text{mut}\right)=1.
  \end{gather*}
\end{lem}


\me
Combining all the previous results, we can prove as in \cite{C06}  that 
 for all $\varepsilon>0$, $t>0$ and
$\Gamma\subset{\cal X}$ measurable,
\begin{equation}
  \label{eq:pf-C6-0}
  \lim_{K\rightarrow+\infty}\mathbb{P}(A_{\varepsilon,n}(t,\Gamma))
  =\mathbb{P}\Big(\mbox{Supp}(\Lambda^\sigma_t)\subset\Gamma\mbox{\ and has $n$
    elements}\Big)
\end{equation}
where the event
\begin{multline*}
  A_{\varepsilon,n}(t,\Gamma):=\Big\{\mbox{Supp}(\nu^K_{t/K\mu_K})\subset\Gamma \mbox{\ has $n$ elements that coexist, say\
  }x_1,\ldots,x_n, \\ \mbox{\ and\ }\forall 1\leq i\leq n,\ |\langle\nu^K_{t/K\mu_k},\mathbf{1}_{\{x_i\}}\rangle
  -\bar{u}_i(\mathbf{x})|<\varepsilon\Big\}
\end{multline*}
and
$(\Lambda^\sigma_t,t\geq 0)$ is the PES defined in Theorem~\ref{thm:PES-fdd}.

\bi 
The proof ends as in \cite{C06}.

\section{About the sign of the fitness functions}
\label{sec:sign-fitness}

In the PES of Theorem~\ref{thm:PES-fdd}, the success of a mutant invasion is governed by the sign of its fitness. Our goal in this
section is to study the fitness of mutant traits in the neighborhood of coexisting traits. These results are used in
Section~\ref{sec:small-mut} to study the local direction of evolution in the PES, and the phenomenon of evolutionary branching.

\me We will assume in all what follows that the traits are one-dimensional, i.e.\ ${\cal X}\subset\RR$.

\begin{prop}
  \label{prop:sign-fitn-1}
  For all $\mathbf{x}^*=(x^*_1,\ldots,x^*_n)\in{\cal D}_n$, we have for all $1\leq i\leq n$, when
  $\mathbf{x}=(x_1,\ldots,x_n)\in{\cal D}_n\rightarrow \mathbf{x}^*$ and $y\rightarrow x^*_i$,
  $$
  f(y;\mathbf{x})= (y-x_i)\left(-d'(x^*_i)+\sum_{k=1}^r
    \frac{g_k\eta'_k(x^*_i)}{1+\sum_{j=1}^n\eta_k(x^*_j)\bar u_j(\mathbf{x}^*)}+o(1)\right).
  $$
\end{prop}

\begin{proof}
  Using the relation $f(x_i;\mathbf{x})=0$, we have
  \begin{align*}
    f(y;\mathbf{x}) & =
    -d(y)+d(x_i)+\sum_{k=1}^r\frac{g_k(\eta_k(y)-\eta_k(x_i))}{1+\sum_j\eta_k(x_j)\bar
      u_j(\mathbf{x})} \\
    & =-(y-x_i)\int_0^1 d'(x_i+(y-x_i)u)du
    +(y-x_i)\sum_{k=1}^r\frac{g_k\int_0^1\eta'_k(x_i+(y-x_i)u)du}{1+\sum_j\eta_k(x_j)\bar
      u_j(\mathbf{x})}.
  \end{align*}
  Proposition~\ref{prop:sign-fitn-1} then easily follows from Lemma~\ref{lem:bar-u-continue}.
\end{proof}

\noindent If we think of $y$ as a mutant trait born from the resident trait $x^*_i$ in the resident population of traits
$\mathbf{x}^*=\mathbf{x}$, this result gives the sign of the fitness function when mutations are small, i.e.\ when $|y-x^*_i|$ small,
except when the first-order term is zero. Note that this first-order term is given by the derivative of the fitness function with
respect to the first variable at $(x^*,x^*)$:
$$
\partial_1 f(x^*_i;\mathbf{x}^*)=-d'(x^*_i)+\sum_{k=1}^r \frac{g_k\eta'_k(x^*_i)}{1+\sum_{j=1}^d\eta_k(x^*_j)\bar u_j(\mathbf{x}^*)}.
$$
Observe also that, for all $1\leq i\leq n$, $f(x_i;\mathbf{x})=0$ for all $x_1,\ldots,x_n$ which coexist. Hence, if $(x_1,\ldots,x_n)$ is in
${\cal D}_n$, we have
$\
\partial_1 f(x_i;\mathbf{x})=-\partial_{i+1} f(x_i;\mathbf{x}).
$

\noindent The differentiability of the fitness function at a point $(y,x_1,\ldots,x_n)\in{\cal X}\times{\cal D}_n$ is ensured by
Lemma~\ref{lem:bar-u-continue}. The following result shows that coexistence of $n+1$ traits $(\mathbf{x},y)$ with $\mathbf{x}$ close
to $\mathbf{x}^*\in{\cal D}_n$ and $y$ close to one of the coordinates $x^*_i$ of $\mathbf{x}^*$ is only possible if $\partial_1
f(x^*_i;\mathbf{x}^*)=0$.

\begin{cor}
  \label{cor:coex}
  If $\mathbf{x}^*\in{\cal D}_n$ and $\partial_1 f(x^*_i;\mathbf{x}^*)\not =0$, then there exists a neighborhood ${\cal V}$ of
  $(\mathbf{x}^*,x^*_i)=(x^*_1,\ldots,x^*_n,x^*_i)$ such that ${\cal V}\cap{\cal D}_{n+1}=\emptyset$.
\end{cor}

\begin{proof}
  Assume that for all neighborhood ${\cal V}$ of $(\mathbf{x}^*,x^*_i)$, ${\cal V}\cap{\cal D}_{n+1}\not=\emptyset$. Since ${\cal
    D}_n$ is an open subset of ${\cal X}^{n}$, there exists $\mathbf{x}$ close enough to $\mathbf{x}^*$ and $y$ close enough to
  $x^*_i$ such that $\mathbf{x}\in{\cal D}_n$, $(x_1,\ldots,x_{i-1},y,x_{i+1},\ldots,x_n)\in{\cal D}_n$, $(\mathbf{x},y)\in{\cal
    D}_{n+1}$ and, by Proposition~\ref{prop:sign-fitn-1}, $f(y;\mathbf{x})>0$ and $f(x_i;x_1,\ldots,x_{i-1},y,x_{i+1},\ldots,x_n)<0$.
  By Theorem~\ref{equilibrium}, this means that $\mathbf{\bar u}(\mathbf{x},y)=(\mathbf{\bar u}(\mathbf{x}),0)$, which contradicts
  the fact that $(\mathbf{x},y)\in{\cal D}_{n+1}$.
\end{proof}

\noindent In the neighborhood of points $\mathbf{x}^*\in{\cal D}_n$ such that $\partial_1 f(x^*_i,\mathbf{x}^*)=0$, we have more
precise results on the sign of the fitness.
\begin{prop}
  \label{prop:sign-fitn-2}
  Fix $\mathbf{x}^*=(x_1^*,\ldots,x_n^*)\in{\cal D}_n$ and $1\leq i\leq n$ such that
  $$
  \partial_1 f(x^*_i;\mathbf{x}^*)=-d'(x^*_i)+\sum_{k=1}^r \frac{g_k\eta'_k(x^*_i)}{1+\sum_{j=1}^d\eta_k(x^*_j)\bar
    u_j(\mathbf{x}^*)}=0.
  $$
  \begin{description}
  \item[\textmd{(i)}] When $(x_1,\ldots,x_n)\rightarrow \mathbf{x}^*$ and $y\rightarrow x^*_i$,
    $$
    f(y;\mathbf{x})=\frac{1}{2}(y-x_i)\left((a+o(1))(y-x^*) +\sum_{j=1}^n(b_j+o(1))(x_j-x^*_j)\right),
    $$
    where
    \begin{gather*}
      a=-d''(x^*)+\sum_{k=1}^r\frac{g_k
        \eta''_k(x^*_i)}{1+\sum_{l=1}^n\eta_k(x^*_l)\bar{u}_l(\mathbf{x}^*)} \, , \
      b_i=a+2\sum_{k=1}^r\frac{g_k\eta'_k(x^*_i) \sum_{l=1}^n\partial_{x_i}(\eta_k(x_l)\bar{u}_l(\mathbf{x}))(\mathbf{x}^*)}
      {\left[1+\sum_{l=1}^n\eta_k(x^*_l)\bar{u}_l(\mathbf{x}^*)\right]^2},
    \end{gather*}
    and for $j\not=i$,
    $$
    b_j=2\sum_{k=1}^r\frac{g_k\eta'_k(x^*_i) \sum_{l=1}^n\partial_{x_j}(\eta_k(x_l)\bar{u}_l(\mathbf{x}))(\mathbf{x}^*)}
    {\left[1+\sum_{l=1}^n\eta_k(x^*_l)\bar{u}_l(\mathbf{x}^*)\right]^2}.
    $$
  \item[\textmd{(ii)}] Assume that for all neighborhood ${\cal V}$ of $(\mathbf{x}^*,x^*_i)$ in ${\cal X}^{n+1}$, ${\cal V}\cap{\cal
      D}_{n+1}\not=\emptyset$. Then, when $\mathbf{x}\rightarrow\mathbf{x}^*$ and $y\rightarrow x^*_i$ such that
    $(\mathbf{x},y)\in{\cal D}_{n+1}$,
    $$
    \bar u_{i}(\mathbf{x},y) + \bar u_{n+1}(\mathbf{x},y) \longrightarrow \bar u_i(\mathbf{x}^*).
    $$
  \item[\textmd{(iii)}] Assume that for all neighborhood ${\cal V}$ of $(\mathbf{x}^*,x^*_i)$ in ${\cal X}^{n+1}$, ${\cal V}\cap{\cal
      D}_{n+1}\not=\emptyset$. Then, when $\mathbf{x}\rightarrow\mathbf{x}^*$, $y\rightarrow x^*_i$ and $z\rightarrow x^*_i$,
    $$
    f(z;\mathbf{x},y)=\frac{1}{2}(z-x_i)(z-y)(a+o(1)).
    $$
  \end{description}
\end{prop}

\noindent If one thinks of $y$ as a mutant born from trait $x_i$ in a coexisting population of traits $\mathbf{x}$, Point~(i) allows to
characterize the cases where $(\mathbf{x},y)\in{\cal D}_{n+1}$ thanks to Proposition~\ref{prop:n-morph}. Point~(ii) shows that, in
this case, the sum of the densities of traits $y$ and $x_i$ is close to $\bar u_i(\mathbf{x}^*)$. Point~(iii) shows that if a second
mutant $z$ is born from either $y$ or $x_i$, then its fitness is positive or negative, depending on the sign of $a$ and on the
position of $z$ with respect to $x_i$ and $y$. If $z$ is a mutant trait born from another resident trait $x_j$ with $j\not=i$, then
the sign of the fitness function can be determined using Proposition~\ref{prop:sign-fitn-1} and Point~(i) of
Proposition~\ref{prop:sign-fitn-2}.

\me\begin{proof}
Since $\partial_1 f(x^*_i;\mathbf{x}^*)=0$, we have
\begin{equation}
  \label{eq:evol-sing}
  d'(x^*_i)=\sum_{k=1}^r\frac{g_k\eta'_k(x^*_i)}{1+\sum_j\eta_k(x^*_j)\bar u_j(\mathbf{x}^*)},  
\end{equation}
and therefore
\begin{align*}
  f(y;\mathbf{x}) & =-(y-x_i)\int_0^1[d'(x_i+(y-x_i)u)-d'(x^*_i)]du \\
  & \qquad+(y-x_i)\sum_{k=1}^r\frac{g_k\int_0^1[\eta'_k(x_i+(y-x_i)u)-\eta'_k(x^*_i)]du} {1+\sum_j\eta_k(x_j)\bar u_j(\mathbf{x})} \\
  & \qquad+(y-x_i)\sum_{k=1}^rg_k\eta'_k(x^*_i) \left(\frac{1}{1+\sum_j\eta_k(x_j)\bar
      u_j(\mathbf{x})}-\frac{1}{1+\sum_j\eta_k(x^*_j)\bar
      u_j(\mathbf{x}^*)}\right) \\
  & =-(y-x_i)(d''(x^*_i)+o(1))\int_0^1(x_i+(y-x_i)u-x^*_i)du \\ & \qquad +(y-x_i)\sum_{k=1}^r\frac{g_k
    (\eta''_k(x^*_i)+o(1))\int_0^1(x_i+(y-x_i)u-x^*_i)du} {1+\sum_j\eta_k(x^*_j)\bar u_j(\mathbf{x}^*)} \\ &
  \qquad+(y-x_i)\sum_{k=1}^r\frac{g_k\eta'_k(x^*_i)}{(1+\sum_j\eta_k(x^*_j)\bar u_j(x^*_j))^2}(\mathbf{x}^*-\mathbf{x})^*
  \left(\sum_{j=1}^n\nabla(\eta_k(x_j)\bar u_j(\mathbf{x}))(x^*)+o(1)\right),
\end{align*}
where the differentiability of $\mathbf{x}\mapsto\bar u(\mathbf{x})$ comes from Lemma~\ref{lem:bar-u-continue}. Point (i) easily follows.

\noindent Point (ii) can be proved as follows. Recall that, by Lemma~\ref{lem:bar-u-continue}, $\bar{\mathbf{u}}$ is a bounded
function on $\mathcal{X}^{n+1}$. Let $u^*_i$ be any accumulation point $\bar u_{i}(\mathbf{x},y) + \bar u_{n+1}(\mathbf{x},y)$ when
$\mathbf{x}\rightarrow\mathbf{x}^*$ and $y\rightarrow x^*_i$ such that $(\mathbf{x},y)\in{\cal D}_{n+1}$. After extracting a
subsequence, we may also assume that for all $j\not=i$, $\bar u_j(\mathbf{x},y)\rightarrow u_j^*$.

\noindent Passing to the limit in~(\ref{eq:n-morph-eq}), we get for all $j\in\{1,\ldots,n\}$,
$$
\sum_{k=1}^r \frac{\eta_{k}(x^*_j)\, g_{k}}{1+ \sum_{m=1}^n\eta_{k}(x^*_m)\, u^*_m} = d(x^*_j).
$$
By Assumption \eqref{A4}, the unique solution of this system of equations is $\mathbf{\bar u}(\mathbf{x}^*)$ and the convergence follows.

\noindent Point (iii) can be proved with a similar computation as for Proposition~\ref{prop:sign-fitn-1}: using~(\ref{eq:n-morph-eq})
and the fact that $f(x_i;\mathbf{x},y)=0$, we have
$$
f(z;\mathbf{x},y) =-(z-x_i)\int_0^1 d'(x_i+(z-x_i)u)du+(z-x_i)\sum_{k=1}^r\frac{g_k\int_0^1\eta'_k(x_i+(z-x_i)u)du}
{1+\sum_{j=1}^n\eta_k(x_j)\bar u_j\mathbf{\mathbf{x},y}+\eta_k(y)\bar u_{n+1}(\mathbf{x},y)}.
$$


Then, using the relation
\[\begin{split}
0=&\frac{z-x_i}{y-x_i}f(y;\mathbf{x},y)=-(z-x_i)\int_0^1 d'(x_i+(y-x_i)u)du
\\
&+(z-x_i)\sum_{k=1}^r\frac{g_k\int_0^1\eta'_k(x_i+(y-x_i)u)du} {1+\sum_{j=1}^n\eta_k(x_j)\bar u_j\mathbf{\mathbf{x},y}+\eta_k(y)\bar
  u_{n+1}(\mathbf{x},y)},
\end{split}\]
we obtain
\begin{align*}
  f(z;\mathbf{x},y) & =-(z-x_i)\int_0^1[d'(x_i+(z-x_i)u)-d'(x_i+(y-x_i)u)]du \\ & \qquad
  +(z-x_i)\sum_{k=1}^r\frac{g_k\int_0^1[\eta'_k(x_i+(z-x_i)u)-\eta'_k(x_i+(y-x_i)u)]du}
  {1+\sum_{j=1}^n\eta_k(x_j)\bar u_j\mathbf{\mathbf{x},y}+\eta_k(y)\bar u_{n+1}(\mathbf{x},y)} \\
  & =-(z-x_i)(z-y)\int_0^1 u\int_0^1 d''(x_i+(y-x_i)u+(z-y)uv)dv\:du \\ & \qquad +(z-x_i)(z-y)\sum_{k=1}^r\frac{g_k\int_0^1
    u\int_0^1\eta''_k(x_i+(y-x_i)u+(z-y)uv)dv\:du} {1+\sum_{j=1}^n\eta_k(x_j)\bar u_j\mathbf{\mathbf{x},y}+\eta_k(y)\bar
    u_{n+1}(\mathbf{x},y)},
\end{align*}
and the result follows from Point~(ii).
\end{proof}

\section{Sketch of the proof of Theorem \ref{thm:br}}
\label{app:br}

We recall the main steps of the proof of the branching criterion in \cite{CM11}. 
\begin{enumerate}
\item Before the first coexistence time, the support of the PES, given by the TSS, reaches $(x^*-\eta,x^*+\eta)$ in finite time
  almost surely (at time $\theta^\sigma_\eta(x^*)$).
\item After time $\theta^\sigma_\eta(x^*)$, the (support of the) PES cannot exit $(x^*-\eta,x^*+\eta)$ while being monomorphic.
\item If $\partial_{11}f(x^*;x^*)<0$, either there is never coexistence, and then the result is proved, or there is coexistence of
  two traits after a finite time, and then, if two traits $x<y$ (say) in $(x^*-\eta,x^*+\eta)$ coexist,
  \begin{enumerate}
  \item by Proposition~\ref{prop:sign-fitn-2}~(iii), the only mutants traits $z$ that can invade the dimorphic population of traits
    $x$ and $y$ are such that $z\in(x,y)$;
  \item the three traits $x$, $y$ and $z$ cannot coexist, and at least one of the resident traits $x$ and $y$ goes extinct, i.e.\
    $\bar u_1(x,y,z)=0$ or $\bar u_2(x,y,z)=0$.
  \end{enumerate}
  This shows that, after any coexistence time, the distance between the two branches can only decrease, until one of the branches
  goes extinct.
\item If $\partial_{11}f(x^*;x^*)>0$,
  \begin{enumerate}
  \item there is almost surely coexistence of two traits in finite time;
  \item by Proposition~\ref{prop:sign-fitn-2}~(iii), if two traits $x<y$ (say) in $(x^*-\eta,x^*+\eta)$ coexist, the only mutant
    traits $z$ that can invade the population of traits $x$ and $y$ are such that $z\not\in[x,y]$;
  \item once such a mutant trait invades, $x$, $y$ and $z$ cannot coexist, and the intermediate trait only goes extinct, i.e.\
    $\bar{\mathbf{u}}(x,y,z)\in\mathbb{R}_+^*\times\{0\}\times\mathbb{R}_+^*$ if $x<y<z$, or
    $\bar{\mathbf{u}}(x,y,z)\in\{0\}\times(\mathbb{R}_+^*)^2$ if $z<x<y$.
  \end{enumerate}
  This shows that the two branches survive until one of them exits of $(x^*-\eta,x^*+\eta)$, and the last point of the proof consists
  in proving that the distance between the two branches becomes bigger than $\eta/2$ in finite time almost surely.
\end{enumerate}
The only steps that require a different proof from Theorem~4.9 of~\cite{CM11} are Steps~3.(b) and 4.(c). Indeed, in~\cite{CM11},
these steps were proved using general results on the long time behavior of 3-dimensional competitive Lotka-Volterra
systems~\cite{zeeman-93}. Here, we do not have such general results, but we have Theorem~\ref{equilibrium}.

\paragraph{Proof of Step~3.(b)}

Let us assume that $\partial_{11}f(x^*;x^*)<0$, and let $x<y\in(x^*-\eta,x^*+\eta)$ coexist, i.e.\ by Corollary~\ref{cor:coex}
$f(x;y)>0$ and $f(y;x)>0$. Let also $z$ be a mutant trait in $(x^*-\eta,x^*+\eta)$ that can invade the resident population of traits
$x$ and $y$, i.e.\ $z\in(x,y)$ and $f(z;x,y)>0$ by Proposition~\ref{prop:sign-fitn-2}~(iii).

\me Using the relation $f(x;y)-f(y;y)=f(x;y)=(x-y)\int_0^1\partial_1 f(y+u(x-y))du$, we have
\begin{equation}
  \label{eq:pf-br}
  \frac{\partial}{\partial x}\left(\frac{f(x;y)}{y-x}\right)=-\int_0^1 u\partial_{11}f(y+u(x-y);y)du.
\end{equation}
Since $\partial_{11}f(x^*;x^*)<0$, this is positive if $x$ and $y$ are sufficiently close to $x^*$. Hence, since $x<z<y$, we have
$f(z;y)/(y-z)>f(x;y)/(y-x)>0$, and thus $f(z;y)>0$. Similarly, $f(z;y)>0$. In addition, by Proposition~\ref{prop:sign-fitn-2}~(iii),
if $x$ and $z$ coexist (i.e.\ if $f(x;z)>0$), then $f(y;x,z)<0$ and if $y$ and $z$ coexist (i.e.\ if $f(y;z)>0$), then $f(x;y,z)<0$.
Therefore, by the characterization of $\bar{\mathbf{u}}(x,y,z)$ of Theorem~\ref{equilibrium}, if $f(x;z)>0$, then
$$
\bar{\mathbf{u}}(x,y,z)=\Big(\bar u_1(x,z),0,\bar u_2(x,z)\Big),
$$
and if $f(y;z)>0$, then
$$
\bar{\mathbf{u}}(x,y,z)=\Big(0,\bar u_1(y,z),\bar u_2(y,z)\Big).
$$
Conversely, in the case when $f(x;z)\leq 0$ and $f(y;z)\leq 0$, then by Theorem~\ref{equilibrium} again
$\bar{\mathbf{u}}(x,y,z)=(0,0,\bar u(z))$.

\me Finally, we see that $x$, $y$ and $z$ never coexist, and that at least one of the traits $x$ or $y$ goes extinct in
$\bar{\mathbf{u}}(x,y,z)$. Therefore, once two traits in $(x^*-\eta,x^*+\eta)$ coexist in the PES, then the next jump in the PES
reaches a position where either only one trait survives in $(x^*-\eta,x^*+\eta)$, or two traits survive in $(x^*-\eta,x^*+\eta)$,
closer to each other than before the jump.\hfill$\Box$\bigskip

\paragraph{Proof of Step~4.(c)}

Let us assume that $\partial_{11}f(x^*;x^*)>0$, and let $x<y\in(x^*-\eta,x^*+\eta)$ coexist, i.e.\ by Corollary~\ref{cor:coex}
$f(x;y)>0$ and $f(y;x)>0$. Let also $z$ be a mutant trait in $(x^*-\eta,x^*+\eta)$ that can invade the resident population of traits
$x$ and $y$, i.e.\ $z\not\in(x,y)$ and $f(z;x,y)>0$ by Proposition~\ref{prop:sign-fitn-2}~(iii). Let us assume for example that
$z<x<y$.

\me Then, it follows from~(\ref{eq:pf-br}) that $f(z;y)>0$. Similarly, since by assumption~(\ref{eq:hyp-evol-br})
$\partial_{22}f(x^*;x^*)>\partial_{11}f(x^*;x^*)>0$,
$$
\frac{\partial}{\partial x}\left(\frac{f(y;x)}{y-x}\right)=-\int_0^1 u\partial_{22}f(y;y+u(x-y))du<0
$$
for all $x$ and $y$ close enough to $x^*$. Therefore, $f(y;x)>0$ implies that $f(y;z)>0$. Hence $y$ and $z$ coexist and by
Proposition~\ref{prop:sign-fitn-2}~(iii) $f(x;y,z)<0$. All these conditions imply by Theorem~\ref{equilibrium} that
$$
\bar{\mathbf{u}}(x,y,z)=\Big(0,\bar u_1(y,z),\bar u_2(y,z)\Big).
$$
Similarly, if $x<y<z$, then
$$
\bar{\mathbf{u}}(x,y,z)=\Big(\bar u_1(x,z),0,\bar u_2(x,z)\Big).
$$
In all cases, $x$, $y$ and $z$ cannot coexist, and once two traits in $(x^*-\eta,x^*+\eta)$ coexist in the PES, then the next jump in
the PES reaches a position where two traits survive, farther from each other than before the jump.\hfill$\Box$

\end{document}